\documentclass[11pt,oneside,reqno]{amsart}
\usepackage{mathpazo} 
\normalfont
\usepackage[T1]{fontenc}

\usepackage[utf8]{inputenc}
\usepackage{amsmath}
\usepackage{amsfonts}
\usepackage{amssymb}
\usepackage{amsthm}
\usepackage{tikz}
\usepackage{tikz-cd}
\usepackage[all]{xy}
\usepackage[margin=1.2in]{geometry}
\usepackage{amscd}
\usepackage[shortlabels]{enumitem}

\DeclareMathOperator{\Hom}{Hom}

\newcommand{\angles}[1]{\left\langle #1 \right\rangle}

\input xy
\xyoption{all}
\theoremstyle{definition}
\newtheorem{mydef}{\textbf{Definition}}[section]
\newtheorem{myeg}[mydef]{\textbf{Example}}

\newtheorem{rmk}[mydef]{\textbf{Remark}}

\theoremstyle{plain}
\newtheorem{mythm}[mydef]{\textbf{Theorem}}

\newtheorem*{nothma}{\textbf{Theorem A}}
\newtheorem*{nothmb}{\textbf{Theorem B}}
\newtheorem*{nothmc}{\textbf{Theorem C}}
\newtheorem*{nothmd}{\textbf{Theorem D}}
\newtheorem*{nothme}{\textbf{Theorem E}}

\newtheorem{lem}[mydef]{\textbf{Lemma}}
\newtheorem{pro}[mydef]{\textbf{Proposition}}

\newtheorem{cor}[mydef]{\textbf{Corollary}}

\newcommand{\T}{\mathbb{T}}
\newcommand{\TT}{\T}

\newcommand{\N}{\mathbb{N}}

\newcommand{\multipleaffil}[3]{%
  \address{%
    \begin{minipage}[t]{\textwidth}
      #1 \\
      #2 \\
      #3 
    \vspace{0.1cm}
    \end{minipage}
  }
}

\begin{document}

\title{Tropical representations and valuated matroids}
\author{Jaiung Jun}
\multipleaffil{Department of Mathematics, State University of New York at New Paltz, NY, USA}{and}{Institute for Advanced Study, Princeton, NJ, USA}
\email{junj@newpaltz.edu, junj@ias.edu}

\author{Kalina Mincheva}
\address{Department of Mathematics, Tulane University, New Orleans, LA 70118, USA}
\email{kmincheva@tulane.edu}

\author{Jeffrey Tolliver}
\address{}
\email{jeff.tolli@gmail.com}
\makeatletter
\@namedef{subjclassname@2020}{%
	\textup{2020} Mathematics Subject Classification}
\makeatother

\thanks{}

\begin{abstract}
We explore several facets of tropical subrepresentations of a linear representation of a group over the tropical semifield $\mathbb{T}$. A key role in the study of tropical subrepresentations is played by two types of modules over a semiring: weakly free and quasi-free modules. We also investigate subgroups of $\text{GL}_n(K)$ for $K=\mathbb{T}$, $ \mathbb{R}_{\geq 0}$, and automorphisms of weakly free modules and tropical prevarieties defined by tropical linear equations. As an application of our results, we provide an intrinsic description of tropical subrepresentation via certain quasi-free modules, and prove that a tropical subrepresentation is equivalent to a valuated matroidal representation. 
\end{abstract}

\maketitle

\tableofcontents

\section{Introduction}

The representation theory of groups studies how a given group can act on a vector space. One may ask what happens if vector spaces are replaced with other mathematical objects such as finite sets, graphs or matroids. For instance, by considering group actions on finite graphs, one can develop an interesting and rich story of Galois theory for finite graphs (see \cite{terras2010zeta}). 

In this paper, we replace vector spaces with their analogues over idempotent semifields. Linear spaces over idempotent semifields are closely related to (valuated) matroids. Some of the more recent results on this topic are  \cite{frenk2013tropical} and \cite{giansiracusa2018grassmann}. 
The relation between linear spaces over idempotent semifields and (valuated) matroids suggests a relation between representation theory over idempotent semifields and representations of a group via its action on matroids. 

A starting point for developing representation theory of groups over an idempotent semifield $K$ is understanding $\text{GL}_n(K)$. It is well-known that the structure of $\text{GL}_n(K)$ is relatively simple as follows
\begin{equation}\label{eq: GL}
 \text{GL}_n(K) \simeq  S_n \ltimes (K^\times)^n,
\end{equation}
where $S_n$ is the symmetric group and $K^\times$ is the multiplicative group of units of $K$. Regardless of the simple structure of $\text{GL}_n(K)$ we obtain non-trivial results.

In \cite{giansiracusa2020matroidal}, Giansiracusa and Manaker introduce a notion of tropical subrepresentation of a linear representation $\rho:G \to \text{GL}_n(K)$ of a group $G$ as a $G$-invariant tropical linear subspace of the space $K^n$ over an idempotent semifield $K$. Equivalently, dimension $d$ tropical subrepresentations of a linear representation are the same as the set of fixed points of the $G$-action on the Dressian $\text{Dr}(d,n)$ (\cite[Corollary 3.0.4]{giansiracusa2020matroidal}).

Let $\mathbb{B}=\{0,1\}$ denote the Boolean semifield. From \eqref{eq: GL}, we see that $\text{GL}_n(\mathbb{B})=S_n$, and so in this case tropical subrepresentations are precisely group homomorphisms $G \to \text{Aut}(M)$ for some matroid $M$ associated to a tropical linear space (\cite[Proposition 3.1.1]{giansiracusa2020matroidal}). 

When one considers the tropical subrepresentations of the regular representation $\mathbb{B}[G]$, the story becomes more interesting. These are precisely the matroids with ground set $G$ for which left-multiplication by each element of $G$ is a matroid automorphism. Giansiracusa and Manaker show that when $G$ is abelian, there is an interesting interplay between number theory and matroid theory. For instance, they prove that if $G$ is abelian, then $G$ is cyclic if and only if the tropical representation $U_{2,G}$ is realizable \footnote{A tropical subrepresentation is said to be \emph{realizable} if it is the tropicalization of a monomial representation.}, where $U_{2,G}$ is the uniform matroid of rank 2 with the underlying set $G$. 

We have a full understanding of low-dimensional tropical representations: In \cite{marcus2024tropical}, Marcus and Phillips classify two-dimensional tropical subrepresentations of $\mathbb{B}[G]$ and in our paper \cite{JMTmatroidsPart1}, we classify three-dimensional tropical subrepresentations of $\mathbb{B}[G]$. 

In order to prove those results in \cite{JMTmatroidsPart1} we develop rudiments of representation theory over semifields. In this paper we aim to further understand tropical representations but we take a module-theoretic point of view. 
We introduce weakly free and quasi-free modules and give examples. We see that some weakly free and quasi-free modules arise very naturally from combinatorial objects (e.g. polyhedral cones) or in the process of tropicalization. We study the automorphism groups of weakly (or quasi-) free modules because tropical subrepresentations turn out to be closely related to understanding the structure of $\text{Aut}(V)$ when $V$ is a tropical linear space. To be precise, if $V$ is a tropical linear space in $K^n$, $\text{Aut}(V)$ is the set of automorphisms $\phi$ of $K^n$ such that $\phi$ and $\phi^{-1}$ preserve $V$. In particular, Aut(V) depends on the embedding of $V$ into an ambient space $K^n$. We later consider a quasi-free module $Q_M$. Using $Q_M$ we can make certain computations of automorphism groups independent of the embedding of $V$.

We also generalize the setting of Giansiracusa and Manaker in \cite{giansiracusa2020matroidal} and their ``matroidal representations'' to ``valuated matroidal representations'' as the combinatorial equivalent of tropical subrepresentations. 

\subsection{Summary of results}

We start by introducing and studying weakly free and quasi-free modules over a semiring $R$ and their properties. These two types of modules play an important role in the study of tropical subrepresentations.

Briefly, a finitely generated module is said to be \emph{weakly free} if there is a unique transition matrix for any two minimal generating sets (Definition \ref{definition: weakly free}). 

On the other hand, elements $x_1,\dots,x_n$ of a module $M$ over a semiring $R$ are said to be \emph{quasi-independent} if an equation of the form
\[
x_i = \sum_j c_jx_j
\]
implies $c_j = \delta_{ij}$. A \emph{quasi-basis} is a quasi-independent set of generators, and $M$ is said to be \emph{quasi-free} of rank $n$ if it has a quasi basis with $n$ elements (Definition \ref{definition: quasi-free}). We prove that any two quasi-bases have the same cardinality (Lemma \ref{definition: quasi-free}), and hence the notion of rank is well-defined for quasi-free modules. One can easily observe that any free module is quasi-free, but not every quasi-free module is free (Example \ref{example: quasi-free but not free}). However, one can show that a module is free if and only if it is quasi-free and projective (Proposition \ref{proposition: free, quasi-free}). 

Quasi-free modules over a zero-sum-free semifield (and hence idempotent semifields or $\mathbb{R}_{\geq 0}$) are weakly free. In particular, in most of the cases that we are interested in, quasi-freeness will be a stronger condition than weak freeness. 

The following characterizes quasi-free modules over $\mathbb{B}$.

\begin{nothma}[Lemma \ref{lemma: atomic}]
Let $M$ be a finitely generated $\mathbb{B}$-module.  Then $M$ is quasi-free if and only if it is atomic when viewed as a lattice. In particular, for any matroid $M$, one obtains a quasi-free module over $\mathbb{B}$ from the lattice of flats of $M$.    
\end{nothma}

Our next result concerns subgroups of $\text{GL}_n(K)$ for an idempotent semifield $K$. We introduce a class of subgroups of $\text{GL}_n(K)$, called \emph{linear subgroups} - these are subgroups of matrices satisfying some linear equations over $K$ (Definition \ref{definition: linear subgroup}). For instance, we prove the following, which will be used for studying valuated matroidal representations.  

\begin{nothmb}[Propositions \ref{proposition: tropical index} and \ref{proposition: nonnegative index}]
Let $\mathbb{T}$ be the tropical semifield and $G\subseteq (\T^\times)^n = \mathbb{R}^n$ be a linear subgroup with a finite defining system of linear equations.  Then, there is some equivalence relation $\sim$ on $\{1,\ldots,n\}$ such that
\[
G = \{x\in\mathbb{R}^n\mid x_i = x_j \;\mathrm{if}\; i\sim j\}.
\]   
A similar result holds for the semifield $\mathbb{R}_{\geq 0}$ and a linear subgroup $G \subseteq (\mathbb{R}_{\geq 0}^\times)^n$.
\end{nothmb}

We call a subspace $V\subseteq \mathbb{R}^n$ \emph{partition subspace} if there exists an equivalence relation $\sim$ on $\{1,2,\dots,n\}$ such that $V=\{x \in \mathbb{R}^n \mid x_i=x_j \textrm{ whenever } i \sim j\}$ (Definition \ref{definition: partition space}). 

Next, we prove one of our key results. Recall that we have the following split short exact sequence of groups (Proposition \ref{proposition: exact sequence R, GL, S_n}): 
\begin{equation}\label{eq: eq seq in intro}
\begin{tikzcd}
1 \arrow[r] & (K^\times)^n \arrow[r,"f"]&
	\text{GL}_n(K)=\text{Aut}(K^n) \arrow[r,"\pi"]
& S_n\arrow[r] & 1
\end{tikzcd}
\end{equation}
where $f$ is the diagonal map and $\pi$ sends a matrix $A$ to the unique permutation $\sigma$ such that $A_{\sigma(i)i}\neq 0$ for all $i$. For any $K$-linear space $V$ in the free module $K^n$, we define $\text{Aut}(V)$ to be the automorphisms $\phi$ of $K^n$ such that $\phi$ and $\phi^{-1}$ preserve $V$. Then $\text{Aut}(V)$ is a linear subgroup of $\text{GL}_n(K)$ (Lemma \ref{lemma: Automorphism groups of tropical spaces are linear subgroups}). In particular the exact sequence \eqref{eq: eq seq in intro} induces a map $\psi:\text{Aut}(V) \to S_n$ as a restriction of $\pi$. 

It is clear that understanding tropical subrepresentations is closely related to understanding the structure of $\text{Aut}(V)$ when $V$ is a tropical linear space. We will further make use of a notion of weak basis lines (Definition \ref{definition: weak base lines}), which can be intuitively considered as ``coordinates'' of weakly free modules.  

With this motivation, suppose $M$ and $\psi$ are as in one of the following cases. 
\begin{enumerate}[label=(\alph*)]
    \item
    Let $M$ be a weakly free module of rank $n$ over $\T$.  Suppose $M$ is finitely presented and that $M$ can be embedded into a finitely generated free $\T$-module.  Let $\psi: \mathrm{Aut}(M)\rightarrow S_n$ be the map describing the action of automorphisms on weak basis lines.
    \item 
    Let $M$ be a weakly free module of rank $n$ over $\mathbb{R}_{\geq 0}$.  Let $\psi: \mathrm{Aut}(M)\rightarrow S_n$ be the map describing the action of automorphisms on weak basis lines.
    \item 
    Let $M\subseteq \T^n$ be a $\T$-linear space.  Suppose $M$ is finitely generated and can be written as an intersection of only finitely many $\T$-hyperplanes.  Let $\psi: \mathrm{Aut}(M)\rightarrow S_n$ be the composition of the inclusion into $\text{GL}_n(\T)$ with the map $\text{GL}_n(\T)\rightarrow S_n$.
\end{enumerate}

\begin{nothmc}
[Theorem \ref{theorem: automorphism group as a semidirect product}]
With the notation as above, we have the following. 
\begin{enumerate}
    \item 
In each of the above cases, we let $H$ be the image of $\psi$.  Then $\mathrm{Aut}(M) \cong H\ltimes V$ for some partition subspace $V\subseteq \mathbb{R}^n$. Moreover, the action of $H$ on $V$ is the permutation action induced by the action of $S_n$ on $\mathbb{R}^n$.\footnote{Note that part of the claim is that $V\subseteq \mathbb{R}^n$ is closed under the permutation action of $H$ on $\mathbb{R}^n$.}  
\item
Furthermore, if $G$ is a finite group, then composition with $\pi: \mathrm{Aut}(M)\rightarrow H$ (where $\pi$ is the corestriction of $\psi$) yields a one-to-one correspondence between equivalence classes of linear $G$-actions on $M$ and homomorphisms $G\rightarrow H$.
\end{enumerate}    
\end{nothmc}

In fact, a similar idea can be applied to a polyhedral cone (viewed as $\mathbb{R}_{\geq 0}$-module). If $M$ is a polyhedral cone and $W$ is the real vector space spanned by $M$, we call a permutation $\sigma$ of the extreme rays of $M$ \emph{realizable} if there is some invertible map $T:W \to W$ such that $T$ maps the $i$th ray to the $\sigma(i)$th ray for all $i$. 

We prove the following. We note that Part (1) is known \cite{horne1978automorphism}, while, to the best of our knowledge, Part (2) is new. Nonetheless, we present them together for clarity and coherence.

\begin{nothmd}[Corollary \ref{corollary: cone case}]
With the same notation as above, let $H$ be the group of realizable permutations of the extreme rays of $M$.  Then the following hold.
\begin{enumerate}
    \item 
    There is a partition subspace $V\subseteq \mathbb{R}^n$ such that the group of invertible linear maps $T: W \rightarrow W$ which preserve $M$ (meaning $T(M) = M$) is isomorphic to $H \ltimes V$ where $H$ acts on $V$ via the permutation action.
    \item 
  There is a one-to-one correspondence between equivalence classes of linear $G$-actions on $W$ which preserve $M$ and homomorphisms $G\rightarrow H$.
\end{enumerate}
\end{nothmd}

Our last theorem is the reinterpretation of tropical subrepresentations as valuated matroidal representations, which generalizes \cite[Proposition 3.1.1]{giansiracusa2020matroidal}. We also consider an alternative representation over $\T$ associated to a valuated matroidal representation in terms of a quasi-free module $Q_M$ associated to a valuated matroid $M$. 

In the following theorem, we highlight that the module $Q_M$ associated to a valuated matroid $M$ does not depend on a choice of an embedding of a tropical linear space in $\mathbb{T}^n$, which was not the case in Giansiracusa and Manaker's work. More precisely, \cite[Remark 3.3.5.]{giansiracusa2020matroidal}  states that tropical representations are seen as invariant tropical linear spaces sitting inside the ambient free module. In particular, the module $Q_M$ may allow one to study the intrinsic module-theoretic structure of representations. 

To introduce our last result we will need the notion of weak automorphism. For a valuated matroid $(M,w)$ on $[n]=\{1,\dots,n\}$, by a weak automorphism of $M$, we mean an element $\sigma \in S_n$ such that there exists a map $\tau:[n] \to \mathbb{T}$ such that for each basis $B$ of the underlying matroid $\underline{M}$, one has
\[
w(\sigma(B))=(\prod_{i \in B}\tau(i))w(B). 
\]

We remark that Jarra, Lorscheid, and Vital consider a more general case of morphisms for matroids over idylls in \cite[Definition 2.2]{jarra2024quiver} in terms of submonomial matrices. The recent paper \cite{iezzi2023tropical} of Iezzi and Schleis on tropical quiver Grassmannians contains related ideas. 

Now let $M$ be a valuated matroid and $V_M$ be a corresponding tropical linear space. The module $Q_M$ is the coordinate ring of the tropical linear space $V_M$, which is a quasi-free module (Lemma \ref{lemma: QM quasi-free lemma}). This allows us to identify $\text{Aut}(Q_M)$, the automorphism group of $Q_M$ as a $\mathbb{T}$-module, with a subgroup of $\text{GL}_n(\mathbb{T})$. With this, we prove the following.

\begin{nothme}[Corollary \ref{corollary: main cor}]
Let $M$ be a valuated matroid and $G$ be a finite group.
\begin{enumerate}
    \item 
The isomorphism classes of tropical subrepresentations whose underlying tropical linear space is isomorphic to $V_M$ are in one-to-one correspondence with weak isomorphism classes of weak $G$-actions on $M$. 
\item 
The image of $\mathrm{Aut}(Q_M)$ under the map $\pi:\text{GL}_n(\T)\rightarrow S_n$ is the weak automorphism group $\mathrm{Aut}_w(M)$.
\end{enumerate}
\end{nothme}

To summarize, for a valuated matroid $M$ with the underlying matroid $\underline{M}$, we have the following diagram:
\[
\begin{tikzcd}[column sep=1.5cm, row sep = 1cm]
G \arrow[r,"tropical"] \arrow[dr,swap,"matroidal"]& \text{Aut}(Q_M)\simeq \text{Aut}_w(M) \arrow[d,"inclusion"]\\
 & \text{Aut}(\underline{M})
\end{tikzcd}
\]
We remark that in general $\text{Aut}_w(M)$ is a proper subgroup of $\text{Aut}(\underline{M})$ even when $M$ is realizable (Example \ref{example: last one}).

\bigskip

\textbf{Acknowledgment} J.J. acknowledges AMS-Simons Research Enhancement Grant for Primarily Undergraduate Institution (PUI) Faculty during the writing of this paper, and parts of this research was done during his visit to the Institute for Advanced Study supported by the Bell System Fellowship Fund. K.M. acknowledges the support of the Simons Foundation, Travel Support for Mathematicians.

\section{Preliminaries}\label{section: preliminaries}

A \textit{semiring} is a set $R$ with two binary operations (addition $+$ and multiplication $\cdot$ ) satisfying the same axioms as rings, except the existence of additive inverses. In this paper, a semiring is always assumed to be commutative. A semiring $(R,+,\cdot)$ is \emph{semifield} if $(R\backslash\{0_R\},\cdot)$ is a group. A semiring $R$ is said to be \emph{zero-sum free} if $a+b=0$ implies $a=b=0$ for all $a,b \in R$. A semiring $R$ is said to be \emph{connected} if any pair $(e,f)$ of elements in $R$ satisfies the conditions $ef=0$ and $e+f=1$, then $\{e,f\}=\{0,1\}$. We call an idempotent semifield $K$ \textit{archimedean} if for every $\lambda, x\in K$ with $x\neq 0$ and $\lambda < 1$, there exists a natural number $n$ such that $\lambda^n \leq x$.

\begin{mydef}
For any totally ordered abelian group $\Gamma$, one can adjoin an extra element $\{-\infty\}$ so that $\overline{\Gamma}:=\Gamma \cup \{-\infty\}$ becomes an idempotent semifield with addition $a+b=\max\{a,b\}$ and multiplication $ab=a+_\Gamma b$. When $\Gamma=\mathbb{R}$, we will use the notation $\mathbb{T}$ instead of $\overline{\mathbb{R}}$. It is called the \emph{tropical semifield}. The subsemifield $\mathbb{B}=\{0,-\infty\}$ is called the \emph{Boolean semifield}.\footnote{Sometimes the elements of $\mathbb{B}$ are written as 0 and 1.} One can easily check that any semifield is connected. 
\end{mydef}

A module over a semiring is defined analogously to modules over a ring, but there is a big difference in how they behave. In our previous paper \cite{JMT20} we remark that a free module over the \emph{Boolean semifield} $\mathbb{B}$ has a unique basis. 
In fact, we prove that ``few bases exist'' over idempotent semirings; see \cite[Proposition 3.15]{JMT20} 
for the precise statement. 

\begin{mydef}
Let $R$ be a semiring and $M$ be a $R$-module. We call $M$ \textit{extremal} if $x+y \in M$ implies $x,y \in M$. 
\end{mydef}

\begin{mydef}
Let $R$ be a semiring and $M$ be a $R$-module. By a \textit{join-irreducible} element of $M$, we mean an element $y \in M$ such that $y=\sum x_i$ implies $y=x_i$ for some $i$.\footnote{The term ``join-irreducible'' comes from lattice theory.}  
\end{mydef}

\begin{mydef}
    A \textit{congruence} on a semiring $R$ is an equivalence relation on $R$ that respects the operations.
\end{mydef}

\begin{mydef}\label{definition: bend relation}
	Let $f$ be a polynomial in the polynomial semiring $\T[x_1,\dots, x_n]$. The \emph{bend relations} of $f$ is the set of equivalences $\{ f\sim f_{\hat{i}}\}$, where $f_{\hat{i}}$ is the polynomial $f$ after removing its $i$-th monomial. For a subset $S$  of $\T[x_1,\dots, x_n]$ the \emph{bend congruence} of $S$ is the congruence generated by the bend relations of all $f \in S$.
\end{mydef}

\begin{pro}\cite[Proposition 3.18]{JMT20}\label{proposition: exact sequence R, GL, S_n}
Let $R$ be a connected zero-sum free semiring.\footnote{This condition holds if $R$ is an idempotent semifield or $\mathbb{R}_{\geq 0}$.}. Then one has the following split short exact sequence of groups which is natural in $R$:
\begin{equation}\label{eq: exact sequence R,GL, S_n}
\begin{tikzcd}
1 \arrow[r] & (R^\times)^n \arrow[r,"f"]&
	\emph{GL}_n(R) \arrow[r,"\pi"]
& S_n\arrow[r] & 1
\end{tikzcd}
\end{equation}
where $f$ is the diagonal map and $\pi$ sends a matrix $A$ to the unique permutation $\sigma$ such that $A_{\sigma(i)i}\neq 0$ for all $i$.    
\end{pro}

In this paper, we will consider three notions of linear spaces:
\[
\{\text{tropical linear spaces}\} \subseteq  \{\text{Bend-$K$-linear spaces}\} \subseteq \{\text{$K$-linear spaces}\}
\]

\begin{mydef}\label{definition: K-linear space}
Let $K$ be a semifield. 
\begin{enumerate}
    \item 
By a \emph{$K$-linear space} $V$ in $K^n$, we mean the subset of $K^n$ obtained in the following way: there exists a congruence relation on $K[x_1,\dots,x_n]$ generated by $\{L_i \sim L_j\}_{i,j \in I}$, where $L_i \in K[x_1,\dots,x_n]$ are linear polynomials, such that 
\[
V=\Hom_K(K[x_1,\dots,x_n]/\sim,K). 
\]
\item 
When $K$ is totally ordered and a congruence relation $\sim$ is a bend relation as in Definition \ref{definition: bend relation}, then $V$ is said to be \emph{Bend-$K$-linear space}.\footnote{In other words, $V$ is the set of points in $K^n$ which ``tropically vanish'' on a set of linear polynomials.} When $K=\mathbb{T}$, this is a tropical prevariety.\footnote{By a tropical prevariety, we mean the set of points that tropically vanish on each of defining equations, but it may not be a tropicalization of an algebraic variety. See \cite[Definition 3.2.1]{maclagan2021introduction} for the case when $K=\mathbb{T}$.}  
\item 
By a \emph{tropical linear space}, we mean a Bend-$K$-linear space whose bend relation comes from circuits of a valuated matroid.\footnote{In much of the tropical literature, one only considers the case when $K=\mathbb{T}$, but in the original paper of Dress and Wenzel \cite{dress1992valuated}, they considered valuated matroids in a more general setting.}
\end{enumerate}
When a Bend-$K$-linear space (resp.~tropical linear space) is defined by the bend relation of a single polynomial, we call $V$ a Bend-$K$-hyperplane (resp.~tropical hyperplane). When a $K$-linear space $L$ is defined by one equation involving two linear expressions, such as $\ell_1 \sim \ell_2$, we call $L$ a $K$-hyperplane.   
\end{mydef}

\begin{rmk}
Since not all $\mathbb{T}$-linear spaces (or even Bend-$\mathbb{T}$-linear spaces) are tropical linear spaces, we use these two terms to distinguish them from tropical linear spaces. We note that Bend-$\T$-hyperplanes are tropical linear spaces as any Bend-$\T$-hyperplane is a tropicalization of a linear space, and hence they are tropical varieties. This is no longer true even when one has only two defining equations. See \cite[Example 3.2.2]{maclagan2021introduction}. In \cite{fink2013tropical}, Fink shows that tropical linear spaces are precisely tropical varieties of degree 1. In particular, Bend-$\T$-hyperplanes are tropical linear spaces. Note that one can also consider that any Bend-$\T$-hyperplane $H$ is the tropicalization of some linear space $L$ over a valued field. Then, the Pl\"ucker coordinate of $L$ in the Grassmannian will impose the tropical linear space structure on $H$. 
\end{rmk}

A tropical linear space is an element in the Dressian, in particular, is a tropical variety. An alternative way to describe the points of a tropical linear space is as the set 
\[
\Hom(\T[x_1, \dots, x_n]/C, \T),
\]
where $C$ is the bend congruence of $S$.

\section{Weakly free and quasi-free modules}\label{section: weakly free and quasi-free section}

In this section we study modules that are not necessarily free but have a finite generating set. Note that even in the case when $R$ is a ring, the existence of a minimal generating set for an $R$-module is not enough to guarantee nice behavior. To motivate the following definition, recall that any two bases of an $R$-module over a ring $R$ are related by an invertible change of basis matrix.

\begin{mydef}\label{definition: weakly free}
We call a finitely generated module $M$ over a semiring $R$ \emph{weakly free} if for any two minimal generating sets $\{x_1,\ldots, x_m\}, \{y_n,\ldots, y_n\}$, we have $m = n$ and that there is a unique invertible matrix $A\in GL_n(R)$ such that $y_i = \sum_j A_{ij} x_j$ for all $i$.  If $M$ is weakly free, then a minimal generating set is called a \emph{weak basis}.
The size of a weak basis is called the rank of $M$.
\end{mydef}

\begin{rmk}
It is convenient for us to include finite generation in the definition of a weakly free module because in this paper we focus on finitely generated modules.  However, for other applications, it is likely more natural to allow infinitely generated modules so that free modules are always weakly free. Generalizing the definition to the infinitely generated case is straightforward.
\end{rmk}

\begin{rmk}\label{rmk: free weakly}
The main case of interest for us is that of semifields.  Over a semifield it is easy to check that free implies weakly free, which justifies the terminology.  But over an arbitrary semiring this is not true because even for a free module, a minimal generating set need not be a basis.  As an example, let $R$ be a semiring in which $1$ is the sum of two non-units (e.g. $R = \mathbb{B} \oplus \mathbb{B}$ or $R = \mathbb{Z}$) which we denote $x, y$.  View $R$ as a module over itself, and observe that $\{x, y\}$ is a minimal generating set which is not a basis.
\end{rmk}

The following result embeds automorphisms of weakly free modules into $\text{GL}_n(R)$.

\begin{lem}\label{lemma: embedding of automorphisms into GL_n}
Let $M$ be a weakly free module over a semiring $R$, and fix a weak basis $x_1, \ldots, x_n$.  Let
\[
\pi: R^n \rightarrow M
\]
be the map sending the standard basis to $x_1,\ldots, x_n$. 
\begin{enumerate}
    \item 
For any $\phi\in \mathrm{Aut}(M)$, there is a unique invertible matrix $A\in \text{GL}_n(R)$ such that
\[
\pi A = \phi\pi.
\]
\item 
Explicitly, the matrix $A$ is the unique invertible matrix such that
\[
\phi(x_i) = \sum_j A_{ji}x_j \quad \forall i.
\]
Moreover, this construction gives a monomorphism $\iota: \mathrm{Aut}(M) \rightarrow GL_n(R)$.
\end{enumerate}

\end{lem}
\begin{proof}First observe that since $\phi$ is an isomorphism, $\phi(x_1),\ldots,\phi(x_n)$ form a minimal generating set for $M$.  By the definition of a weakly free module, there is a unique $A\in GL_n(R)$ such that $\phi(x_i) = \sum_j A_{ji}x_j$ for all $i$. 

Next we observe that
\begin{equation}
\pi (A e_k) = \pi(\sum_i A_{ik} e_i) = \sum_i A_{ik} x_i = \phi(x_k) = \phi\pi(e_k).
\end{equation}
Since this holds for each basis vector $e_k$, we get $\pi A = \phi \pi$.  Reversing this argument shows that the equation $\pi A = \phi \pi$ implies $\phi(x_i) = \sum_j A_{ji}x_j$, so $A$ is the unique solution to this equation.

Now let $\phi, \psi$ be automorphisms of $M$ and let $A=\iota(\phi)$ and $B=\iota(\psi)$.  Then 
\begin{equation}
\pi AB = \phi\pi B = \phi \psi \pi
\end{equation}
so $\iota(\phi\psi) = AB = \iota(\phi)\iota(\psi)$, so $\iota$ is indeed a homomorphism.

Finally, if $\phi \in \ker \iota$, then $\pi = \phi\pi$.  Since $\pi$ is an epimorphism, this implies $\phi$ is the identity.  Thus $\iota$ is injective.
\end{proof}

We can easily show that weak bases over connected zero-sum-free semirings are unique up to permutation and rescaling.

\begin{lem}\label{lemma: weak basis lines}
Let $R$ be a connected zero-sum-free semiring and $M$ be a weakly free module.  Let $x_1,\ldots, x_n$ and $y_1,\ldots, y_m$ be weak bases.  Then there is a permutation $\sigma\in S_n$ along with units $\lambda_i\in R^\times$ such that $y_i = \lambda_i x_{\sigma(i)}$ for all $i$.
\end{lem}
\begin{proof}This follows immediately from the definition and the fact that $GL_n(R)$ consists of products of an invertible diagonal matrix and a permutation matrix.
\end{proof}

We will later need the following definition which is heavily used in \cite{JMTmatroidsPart1} to classify indecomposable representations of a group $G$ over an idempotent semifield. 

\begin{mydef}\label{definition: weak base lines}
Let $V$ be a weakly free module over an idempotent semifield. Pick a weak basis $B=\{v_1,\dots,v_n\}$ for $V$. By the set $S$ of\textit{ weak basis lines}, we mean the submodules of $V$ spanned by each $v_i$, i.e., 
\[
S=\{\text{span}(v_i) \mid v_i \in B\}. 
\]
\end{mydef}

Note that from Lemma \ref{lemma: weak basis lines} the set of weak basis lines of a weakly free module does not depend on a choice of basis.

\begin{rmk} 
The previous lemma can be seen as an analogue to our classification results in \cite{JMTmatroidsPart1} for representations whose underlying module is free since it essentially says that automorphisms preserve the set of weak basis lines. In particular, under the conditions of the lemma, a linear $G$-action on $M$ induces an action on the set of weak basis lines.
\end{rmk}

An alternative approach to studying actions on weakly free modules is to lift them to actions on free modules, which we already understand. In fact, the following is immediate from Lemma \ref{lemma: embedding of automorphisms into GL_n}.

\begin{lem}
Let $G$ be a group. Then, with the same notation as in Lemma \ref{lemma: embedding of automorphisms into GL_n}, a linear $G$-action $\rho: G\rightarrow \mathrm{Aut}(M)$ on $M$ can be uniquely lifted to a linear $G$-action $\tau$ on $R^n$ such that $\pi\tau(g) = \rho(g)\pi$ for all $g$.
\end{lem}

We have not yet shown that there are interesting examples of weakly free modules.  One simple class of examples is finitely generated $\mathbb{B}$-modules.
\begin{myeg}\label{example: B-modules are weakly free}
Let $M$ be a finitely generated (or equivalently finite) $\mathbb{B}$-module, then $M$ is weakly free and the set of join-irreducible elements is a weak basis. Note that finite $\mathbb{B}$-modules do not have to be free. See \cite[Example 3.21]{JMT20} for instance. 

To see the claim, let $x_1,\ldots, x_n$ be a generating set and $y$ be join-irreducible.  We have a basis expansion
\begin{equation}
y = \sum_{i\in S} x_i,
\end{equation}
which implies that $x_i = y$ for some $i$.  So every generating set contains the set of join irreducible elements.  
A standard inductive argument implies that every element is a sum of join-irreducible elements, so they form a generating set.  By the above, the set of join-irreducible elements is the unique minimal generating set, which implies the claim.
\end{myeg}

Another example of weakly free modules is given in \cite{wagneur1991moduloids}. There the author shows that finitely generated submodules of finitely generated free modules over $\mathbb{T}$ are weakly free. Earlier studies of modules and the relaxation of freeness can be found in \cite{moller1988theorie} and \cite{gaubert1997methods} and literature referenced there.\footnote{In \cite{gaubert1997methods}, Max Plus is a collective name for a working group on $(\text{max}, +)$ algebra.}

The proof of the following statement appears in \cite{wagneur1991moduloids}, but we include a sketch as our terminology differs significantly from Wagneur's. 

\begin{pro}\cite[Theorem 5]{wagneur1991moduloids}\label{proposition: Wagneurs result} 
Let $K$ be a totally ordered archimedean idempotent semifield (such as $\T$) and $M\subseteq K^n$ be a finitely generated submodule.  Then $M$ is weakly free.
\end{pro}
\begin{proof}
 We can first check that if $x$ is an element of a minimal generating set $S$ and $y$ is a linear combination of the other generators in $S$, then $x = \lambda x + y$ implies $\lambda = 1$ as follows.  Clearly $x \geq \lambda x$, so by focusing on some nonzero coordinate we get $\lambda \leq 1$.  But if $\lambda < 1$, we would have 
\begin{equation}
x = \lambda^n x + (y + \lambda y + \ldots + \lambda^{n-1} y) = \lambda^n x + y    
\end{equation}
for any $n$ and the archimedean property yields $x = y$, which contradicts minimality.

Suppose we have an equation of the form $x = y + z$ with $x$ belonging to the minimal generating set.  Write $y = \alpha x + u$ and $z = \beta x + v$ with $u, v$ being linear combinations of the other generators besides $x$.  Then $x = (\alpha + \beta) x + (u + v)$ which implies $\alpha + \beta = 1$, which in turn implies $\alpha = 1$ or $\beta = 1$.  In the first case $y \leq y + z = x \leq x + u = y$, so $x = y$, and the other case is similar. Therefore, each element of the minimal generating set is join-irreducible.

Now given two minimal generating sets $x_1,\ldots, x_m$ and $y_1,\ldots, y_n$, write $x_i = \sum a_{ij} y_j$ and by join-irreducibility, there is some $j$ such that $x_i = a_{ij}y_j$.  It is now easy to check the minimal generating set is unique up to permutation and rescaling.
\end{proof}

The following is another notion related to free modules that we will consider.

\begin{mydef}\label{definition: quasi-free}
Let $M$ be a module over a semiring $R$ and let $x_1,\ldots, x_n\in M$.  We say $x_1, \ldots, x_n$ are \emph{quasi-independent} if an equation of the form $x_i = \sum_j c_j x_j$ implies $c_j = \delta_{ij}$. A \emph{quasi-basis} is a quasi-independent set of generators.  $M$ is said to be \emph{quasi-free} of rank $n$ if it has a quasi-basis with $n$ elements. Quasi-free will always mean quasi-free of finite rank unless otherwise specified.
\end{mydef}

The following result shows that any two quasi-bases are related by an invertible matrix.  There may be other minimal generating sets, so this does not imply that a quasi-free module is weakly free, but is good enough for many applications of the weak freeness criterion.

\begin{lem}\label{lemma: quasi-basis permute}
Let $x_1,\ldots,x_n$ and $y_1,\ldots, y_m$ be two quasi-bases for a quasi-free $R$-module $M$ over a semiring $R$.  Then $m = n$ and there is a unique invertible matrix $A$ such that $y_i = \sum_j A_{ij} x_j$ for all $i$.
\end{lem}
\begin{proof}
Invertible matrices must be square as a result of \cite[Proposition 3.5]{JMT20}.  So $m=n$ follows from the second part of this lemma.

Since $y_1, \ldots, y_m$ are generators, there exists a collection of elements $B_{ij}\in R$ such that
\begin{equation}
x_i = \sum_j B_{ij} y_j.
\end{equation}
Similarly we may choose $A_{ij}$ such that 
\begin{equation}
y_i = \sum_j A_{ij} x_j.
\end{equation}  We will show $A = B^{-1}$, which implies both the invertibility claim and the uniqueness claim. Plugging one of the above equations into the other yields
\begin{equation}
x_i = \sum_{j,k} B_{ij} A_{jk} x_k.
\end{equation}
Applying quasi-independence yields
\begin{equation}
\sum_{j} B_{ij} A_{jk} = \delta_{ik}
\end{equation}
 or in other words, $BA$ is the identity.  That $AB$ is the identity follows similarly.
\end{proof}

Quasi-free modules over zero-sum-free semifields are weakly free. 

\begin{lem}\label{lemma: quasi-free implies weakly free}
Let $M$ be a quasi-free module over a zero-sum-free semiring $R$.  Let $x_1,\ldots,x_n$ be a quasi-basis.
\begin{enumerate}
\item 
If $x_k = y + z$ for some $k$ and some $y, z\in M$ then there exist $a,b\in R$ such that $y=ax_k$ and $z=bx_k$.
\item 
If we assume in addition that $R$ is a semifield, then a subset of $M$ is a quasi-basis if and only if it is a minimal generating set.  In particular $M$ is weakly free.
\end{enumerate}
\end{lem}
\begin{proof}
For the first part, write $y$ and $z$ in terms of the quasi-basis as $y = \sum_i a_i x_i$ and $z = \sum_i b_i x_i$.  Then $x_k = \sum_i (a_i + b_i) x_i$.  By quasi-independence, $a_i + b_i = 0$ for $i \neq k$, and by the zero-sum-free condition, $a_i = b_i = 0$ for $i \neq k$.  So $y = a_k x_k$ and $z = b_k x_k$.

For the second part, we first observe that a quasi-basis is a minimal generating set, because if we could remove some element of the quasi-basis (say $x_k$), then we would have a relation of the form
\begin{equation}x_k = \sum_i a_i x_i
\end{equation}
with $a_k = 0$.  But quasi-independence implies $a_k = 1$.

Now let $y_1,\ldots, y_m$ be a minimal generating set.  For each $i$, we can write $x_i$ in terms of this generating set in the form
\begin{equation}\label{eq:xy}
x_i = \sum_{j} a_{ij}y_j.
\end{equation}
By the first claim, each term in (\ref{eq:xy}) is a multiple of $x_i$, say
\begin{equation}a_{ij}y_j = b_{ij}x_i. 
\end{equation}
Furthermore it is clear that $a_{ij}\neq 0$ for at least one $j$; call this value $j=\eta(i)$.  Then we have $y_{\eta(i)} = a_{i\eta(i)}^{-1}b_{i\eta(i)}x_i$.  So $y_{\eta(i)}$ is a multiple of $x_i$ (and it is nonzero, since it is part of a minimal generating set).  
$\eta$ is injective since $\eta(i)=\eta(j)$ would imply $x_i$ is a multiple of $x_j$.  The set $\{ y_{\eta(i)} \mid i \in \{1,\ldots,n\}\}$ is a generating set because each element of the generating set $\{x_1,\ldots,x_n\}$ lies in its span.  By minimality of $y_1,\ldots,y_m$, $\eta$ is surjective, so is a permutation.  
The equation $y_{\eta(i)} = a_{i\eta(i)}^{-1}b_{i\eta(i)}x_i$ now says that a minimal generating set agrees with the chosen quasi-basis up to permutation and rescaling by units.  But this implies $y_1,\ldots,y_m$ is also a quasi-basis.  The remaining claim that $M$ is weakly free follows because any two minimal generating sets are quasi-bases and therefore agree up to a unique invertible matrix by Lemma \ref{lemma: quasi-basis permute}.
\end{proof}

As a first example of a quasi-free module, we show that atomic lattices are quasi-free.

\begin{lem}\label{lemma: atomic}
Let $M$ be a finitely generated $\mathbb{B}$-module.  Then $M$ is quasi-free if and only if it is atomic when viewed as a lattice.
\end{lem}
\begin{proof}
We first note that it is well-known that atomic lattices are characterized by the property that join-irreducible elements are atoms.  In addition, Example \ref{example: B-modules are weakly free} shows that the join-irreducible elements are the unique minimal generating set.

Suppose $M$ is quasi-free.  Then the quasi-basis $x_1,\ldots,x_n$ is a minimal generating set by Lemma \ref{lemma: quasi-free implies weakly free}, and hence is the set of join-irreducible elements; we must show each element is an atom; without loss of generality we focus on $x_1$.  Suppose $y\in M$ is such that $y\leq x_1$.  We can choose $c_i\in\mathbb{B}$ such that $y = \sum_i c_i x_i$.  Then $x_1 = y + x_1 = x_1 + \sum_i c_i x_i$.  By quasi-independence, $c_i = 0$ for $i \neq 1$.  And then either $y = 0$ or $y = x_1$ depending on $c_1$.

Conversely, suppose $M$ is atomic.  Let $x_1,\ldots,x_n$ be the join-irreducible elements (these are atoms by assumption).  Suppose $x_i = \sum_j c_j x_j$ for some $i$ and some collection of elements $c_j\in\mathbb{B}$.  Then for all $j$, $c_j x_j \leq x_i$.  Because $x_i$ is an atom, $c_j x_j = x_i$ which is impossible unless either $c_j = 0$ or $j = i$.  Thus only the $i$th coefficient is nonzero, so $x_i = c_i x_i$ which implies $c_i = 1$.  Thus $x_1,\ldots,x_n$ form a quasi-basis.
\end{proof}

The main advantage of considering quasi-free modules rather than weakly free modules is that it is often easy to show quotient modules are quasi-free.  To this end, we will need the following lemma.

\begin{lem}\label{lemma: equivalence relation identifying vectors with two nonzero components}
Let $M$ be a module over a zero-sum-free semiring $R$ without zero-divisors.Suppose $M$ has a quasi-basis $x_1, \ldots, x_n$.  Let $\equiv$ be the equivalence relation such that $x \equiv y$ if $x = y$ or if $x$ and $y$ can both be written in the form $\sum c_i x_i$ with at least two $c_i$ being nonzero.  Then $\equiv$ is a congruence.
\end{lem}
\begin{proof}
One easily check that $\equiv$ is an equivalence relation. It is also clear that $x\equiv y$ implies $rx\equiv ry$ for $r\in R$ and $x, y\in M$ since $R$ does not have any zero-divisors. 

It remains to show $x\equiv y$ implies $(x + z)\equiv (y + z)$ for $x, y, z\in M$.  If $x = y$, then there is nothing to show.  Otherwise, we may write $x = \sum a_i x_i$ with at least two $a_i$ being nonzero, and we obtain a similar description of $y$ with coefficients $b_i$.  Since $x_1,\ldots,x_n$ form a generating set, we may write $z = \sum c_i x_i$ for some $c_i$.  Then $(x + z) = \sum (a_i + c_i) x_i$.  By the zero-sum-free property, if $c_i \neq 0$ then $a_i + c_i \neq 0$, and so this occurs at least twice.  Similarly, we may write $(y + z)$ in terms of the quasi-basis with two nonzero coefficients.  So, $(x + z)\equiv (y+z)$ as desired.
\end{proof}

We may now prove that certain quotient modules of quasi-free modules are quasi-free.

\begin{pro}\label{proposition: quasi-free quotient}
Let $M$ be a module over a zero-sum-free semiring $R$ without zero-divisors.  Suppose $M$ has a quasi-basis $x_1, \ldots, x_n$.  Let $\sim$ be a congruence on $M$ generated by a collection of pairs $(a_i, b_i)$ such that none of the $a_i$ or $b_i$ is a scalar multiple of a quasi-basis vector.  Then $M / \sim$ is quasi-free and the images of $x_1,\ldots,x_n$ form a quasi-basis.
\end{pro}
\begin{proof}
Let $\equiv$ be as in Lemma \ref{lemma: equivalence relation identifying vectors with two nonzero components}.  The hypothesis implies all generators of $\sim$ belong to $\equiv$.  Since $\equiv$ is a congruence, we see that $a \sim b$ implies  $a\equiv b$ for any $a,b \in M$.

It is clear that the equivalence classes of $x_1,\ldots,x_n$ generate $M/ \sim$.  To see they are quasi-independent, suppose we have an equation of the form
\begin{equation}
x_i \sim \sum c_j x_j.
\end{equation}
Then we also have
\begin{equation}
x_i \equiv \sum c_j x_j.
\end{equation}
But, by quasi-independence in $M$, $x_i$ cannot be written as $\sum a_j x_j$ with at least two $a_j$ being nonzero.  So $x_i \equiv \sum c_j x_j$ implies $x_i = \sum c_j x_j$.  By quasi-independence in $M$, we obtain $c_j = \delta_{ij}$, which implies the result.
\end{proof}

\begin{myeg}\label{ex:monoid-val-matroid}
In this example we show how we tropicalize a $\mathbb{K}$-vector space by interpreting it as a quotient space of $\mathbb{K}^n$, where $\mathbb{K}$ is a valued field, to obtain a quasi-free $\mathbb{T}$-module. 

For a vector $v\in \mathbb{T}^n$, we may define the associated bend relations on $\mathbb{T}^n$ as the relations
\begin{equation}
\sum_{j\neq i}v_je_j \sim \sum_jv_je_j
\end{equation}
for all $i$, where $e_i$ is the standard basis vector.  Given a subset $U\subseteq \T^n$, we obtain the associated bend congruence generated by the bend relations associated to each element of $U$. We let $\mathbb{T}^n/\sim_U$ be the quotient $\mathbb{T}$-module obtained in this way. 

To be specific, one may consider the following. Pick a field $\mathbb{K}$ with a valuation $\nu:\mathbb{K} \to \mathbb{T}$, a finite dimensional $\mathbb{K}$-vector space $V$, and a spanning set $x_1,\ldots, x_n$ of $V$. Then, we can view $V=\mathbb{K}^n/U$, where $U=\{(a_1,\dots,a_n) \in \mathbb{K}^n \mid \sum_i a_ix_i =0\}$.

Recall that tropicalization depends on an embedding of an algebraic variety into an ambient space. Here, we tropicalize $V$ with respect to the fixed isomorphism $V\simeq\mathbb{K}^n/U$ by considering the quotient of $\T^n$ by the bend relations associated to the following set
\[
\nu(U):=\{ (\nu(a_1),\ldots,\nu(a_n))\in \T^n \mid (a_1,\dots,a_n) \in U \}.
\]

If no two elements of the spanning set are $\T$-multiples of each other, then each 
\[
(\nu(a_1),\ldots,\nu(a_n))\in \T^n
\]
has at least 3 nonzero components. Then, the quotient module $\T^n/\sim_{\nu(U)}$ is quasi-free (viewed as $\T$-module), with the equivalence classes of the standard basis forming a quasi-basis. We say $\T^n/\sim_{\nu(U)}$ is the tropicalization of $V$ with respect to the presentation $V \simeq \mathbb{K}^n/U$. 

Note that in Section \ref{section: valuated matroids}, we will revisit this example in the following context: let $Q_M=\T^n/\sim_{\nu(U)}$, then $Q_M^\vee:=\Hom(Q_M,\T)=V_M$, where $V_M$ is the tropicalization of $V$ and $M$ is the corresponding valuated matroid. This line of ideas has been intensively studied in \cite{giansiracusa2016equations} and \cite{giansiracusa2018grassmann}.

One can apply this construction to infinite dimensional spaces as well. Let $V$ be the coordinate ring of an affine variety $X$ as an infinite dimensional vector space spanned by monomials. Applying our construction above to $V$ will give the coordinate semiring of the (scheme-theoretic) tropicalization of $X$. 
\end{myeg}

There is also a partial converse to the above proposition.

\begin{pro}\label{proposition: characterization of quasi-free modules}
Let $M$ be a {quasi-free} module over a zero-sum-free semifield $K$.  Fix a quasi-basis of $M$.  Then there is a canonical isomorphism $K^n / \sim \rightarrow M$, where $\sim$ is some subcongruence of the congruence $\equiv$ in Lemma \ref{lemma: equivalence relation identifying vectors with two nonzero components}.  This isomorphism maps the classes of the standard basis vectors to the quasi-basis.
\end{pro}
\begin{proof}
As the quasi-basis is a generating set, there is an epimorphism $K^n \rightarrow M$ mapping the basis of $K^n$ to the quasi-basis of $M$. This descends to an isomorphism $K^n / \sim \rightarrow M$ for some congruence $\sim$; we need to show that $a\sim b $ implies $a\equiv b$ for $a,b \in K^n$.

Suppose we have a relation of the form
\begin{equation}
c e_i \sim \sum_j c_j e_j
\end{equation}
with $c$ nonzero.  Then $e_i \sim \sum_j c^{-1} c_j e_j$.  As the classes of $\{e_i\}$ form a quasi-basis of $K^n / \sim$, this implies $c^{-1} c_j = \delta_{ij}$, so $c_j = 0$ for $j \neq i$ and $c_i = c$.  Thus the relation above is simply $ce_i \sim ce_i$.

Given any relation $\sum a_i e_i \sim \sum b_i e_i$, there are 3 cases. First case: both sides have at least two nonzero terms, then $\sum a_i e_i \equiv \sum b_i e_i$. Second case: either side has a single nonzero term, then we saw the relation is simply $ce_i \sim ce_i$, and of course $ce_i \equiv ce_i$. Third case: all terms are zero, and so the relation is $0\sim 0$, which of course holds for $\equiv$ as well. Thus we conclude that $\sim$ is a subcongruence of $\equiv$.
\end{proof}

There is a partial converse to the result of Lemma \ref{lemma: quasi-free implies weakly free} that the 1-dimensional submodules spanned by quasi-basis vectors are extremal. 

\begin{lem}\label{lemma: modules generated by extreme rays are quasi-free}
Let $K$ be a zero-sum-free semifield and $M$ be a finitely generated $K$-module. Suppose $M$ has a finite generating set $x_1,\ldots,x_n$ such that the submodules $Kx_i$ are extremal in the sense that $x + y\in Kx_i$ implies $x, y\in Kx_i$. Then the following hold.  
\begin{enumerate}
    \item 
$M$ is weakly free. 
    \item 
Suppose in addition that $M$ has the property that for $a,b\in K$ and $0\neq x\in M$, $ax=bx$ implies $a=b$.  Then $M$ is quasi-free, and if the specified generating set is minimal then $x_1,\ldots, x_n$ is a quasi-basis.   
\end{enumerate}
\end{lem}
\begin{proof}
The first assertion follows by the same argument as \cite[Theorem 1]{wagneur1991moduloids}.  Without loss of generality, we may assume $x_1,\ldots, x_n$ is a minimal generating set.  Fix another minimal generating set $\{ y_i \}_{i\in I}$.  Then for any $k$, there are $a_{ki}\in K$
such that 
\begin{equation}
x_k = \sum a_{ki} y_i.
\end{equation}
By the extremal property, each nonzero term belongs to $Kx_k$.  So there is some index $\eta(k)$ such that $y_{\eta(k)} = c_kx_k$ for some $c_k\neq 0$ (if there are multiple such elements, we choose one).  Of course $\eta(i) = \eta(j)$ implies $x_i$ is proportional to $x_j$, which by minimality implies $i = j$.  Moreover, the submodule generated by $\{y_{\eta(1)},\ldots,y_{\eta(n)}\}$ contains the generating set $x_1,\ldots, x_n$.  So by minimality of $\{y_i\}_{i\in I}$, $\eta$ is surjective.  We have shown $y_{\eta(k)} = c_k x_k$ for some units $c_k$ and some permutation $\eta$.  This implies the minimal generating set is unique up to permutation and rescaling, which implies the result.

For the second assertion, we may assume that $x_1,\ldots, x_n$ is a minimal generating set (otherwise, we may delete some elements, and the hypotheses still hold).  Suppose we have a relation of the form
\begin{equation}
x_k = \sum c_i x_i.
\end{equation}
Then the extremal property of $Kx_k$ implies each term is a multiple of $x_k$, say $c_ix_i = ax_k$.  If $c_i\neq 0$ for some $i\neq k$, we would have $x_i = c_i^{-1} a x_k$, contradicting minimality.  So $c_i = 0$ for $i\neq k$.  So $x_k = c_k x_k$, which by hypothesis implies $c_k = 1$ which in turn implies $x_1,\ldots,x_n$ is a quasi-basis. 
\end{proof}

We use this condition to give another example of quasi-free modules. We consider a polyhedral cone inside $\mathbb{R}^n$ as an $\mathbb{R}_{\geq 0}$-module.  We recall the following conical version of the Krein-Milman theorem, in the particularly simple case of polyhedral cones.  For a proof using semiring theory, see \cite[Corollary 8.7]{borger2024facets}.

\begin{lem}\label{lemma: conic Krein-Milman}
Let $V$ be a real vector space, and let $M\subseteq V$ be a polyhedral cone which does not contain a line.  Then $M$ is generated by its extreme rays, i.e., by those rays $\rho$ such that if $x+y\in \rho$ then $x, y\in \rho$.
\end{lem}

As a corollary, we obtain that such cones are quasi-free $\mathbb{R}_{\geq 0}$-modules.
\begin{cor}Let $V$ be a real vector space, and let $M\subseteq V$ be a polyhedral cone which does not contain a line.  Then $M$ is a finitely generated quasi-free $\mathbb{R}_{\geq 0}$-module and a quasi-basis consists of exactly one element from each extreme ray.
\end{cor}
\begin{proof}
Pick a set $x_1,\ldots, x_n$ consisting of one nonzero element from each extreme ray.  By Lemma \ref{lemma: conic Krein-Milman}, $x_1,\ldots, x_n$ generate $M$.  If $x_k = \sum_{i\neq k} c_ix_i$, then some term is nonzero, and by extremality for this term we have $ax_k = c_i x_i$ for some $a \neq 0$. This contradicts the fact that the generating set contains only one element from each ray.  Thus $x_1,\ldots, x_n$ is a minimal generating set.  The result follows from Lemma \ref{lemma: modules generated by extreme rays are quasi-free}.
\end{proof}

We now provide a simple example showing how knowing that a module is quasi-free or weakly free helps us classify the automorphisms or linear $G$-actions on the module, which is our original motivation. 

\begin{myeg}
Let $M$ be the $\T$-module generated by three elements $x_1, x_2, x_3$ and the relation
\[
x_1 + x_2 = x_1 + x_3.
\]
Since both sides of the relation have at least two terms, $M$ is quasi-free, and hence any automorphism can only rescale or permute the quasi-basis $x_1, x_2, x_3$.  

If we quotient $M$ by the relation $a x_2 = 0$ for some $a\neq 0$, then we obtain the module generated by $x_1, x_3$ with $x_1 = x_1 + x_3$.  It is easy to check that this is not the free module on one-generator.  On the other hand, if we quotient $M$ by the relation $x_1 = 0$, we obtain the module generated by $x_2, x_3$ with relation $x_2 = x_3$, which is free on one generator.  So no automorphism can map $x_1$ to a multiple of $x_2$, and similarly no automorphism can map $x_1$ to a multiple of $x_3$.  Hence any automorphism maps $x_1$ to a multiple of $x_1$.

Now $x_2$ then cannot be mapped to a multiple of $x_1$ so either it is mapped to a multiple of $x_2$ (in which case $x_3$ is mapped to a multiple of $x_3$) or it is mapped to a multiple of $x_3$ (so $x_3$ is mapped to a multiple of $x_2$).

Consider again the quotient by the relation $x_1 = 0$, which we now know to be invariant under automorphism. Any automorphism of $M$ descends to an automorphism of the module generated by two elements $x_2, x_3$ with relation $x_2 = x_3$.  In particular, if $x_2, x_3 \mapsto bx_2, cx_3$ or if $x_2, x_3\mapsto bx_3, cx_2$, we obtain $b = c$.

So any automorphism either has the form
\[
x_1, x_2, x_3 \mapsto ax_1,~bx_2,~bx_3
\]
or
\[
x_1, x_2, x_3 \mapsto ax_1,~bx_3,~bx_2
\]
for some $a, b\in \mathbb{T}^\times$.  It is easy to check that anything of this form is indeed an automorphism.

If we wish to study linear actions of a finite group on $M$, then only the torsion automorphisms are of interest.  In this case, it is clear that $a = b = 1$.  So linear actions of a finite group $G$ on $M$ are equivalent to homomorphisms $\chi: G\rightarrow \mathbb{Z} / 2\mathbb{Z}$ where $g\in G$ acts as either the identity or by swapping $x_2$ and $x_3$.
\end{myeg}

To finish the discussion on modules, we note that quasi-free projective modules are free. The proof is based on \cite[Theorem 8.8]{borger2024facets}.

\begin{pro}\label{proposition: free, quasi-free}
Let $M$ be a module over a semiring $K$. Then, $M$ is quasi-free and projective if and only if $M$ is free. 
\end{pro}
\begin{proof}
We note that any a free module over a semiring does not have to be weakly free (Remark \ref{rmk: free weakly}), but it is quasi-free. Any free module is projective. So, one direction is clear. For the converse, fix a quasi-basis $x_1,\ldots, x_n$\footnote{Despite the notation, close inspection will reveal this works for modules with an infinite quasi-independent generating set, rather than just the finitely generated case.}.  Let $p: K^n \rightarrow M$ map the standard basis to $x_1,\ldots, x_n$.  This is surjective because $x_1,\ldots, x_n$ is a generating set.  By projectivity there exists a section $s: M \rightarrow K^n$.  Define $a_{ij}$ by $s(x_i) = \sum a_{ij} e_j$.  Then applying $p$ to both sides gives
\begin{equation}
x_i = p(s(x_i)) = \sum a_{ij} x_j.
\end{equation}
By quasi-independence, $a_{ij} = \delta_{ij}$.  Thus $s(x_i) = e_i$.  Now $s(p(e_i)) = s(x_i) = e_i$ so $s \circ p = \mathrm{id}$. Moreover, by definition of $s$, we know $p \circ s = \mathrm{id}$.  Thus $p: K^n\rightarrow M$ is an isomorphism.
\end{proof}

The following example shows that the class of quasi-free modules is strictly bigger than the class of free modules. 

\begin{myeg}\label{example: quasi-free but not free}
Consider the lattice of flats of the uniform matroid $U_{1,3}$.
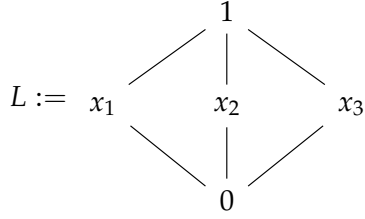
\[
L:=\begin{tikzcd}[every arrow/.append style={dash}]
	& 1  \arrow[d] \arrow[dl] \arrow[dr] &\\
	x_1 \arrow[dr] & x_2 \arrow[d] & x_3 \arrow[dl]\\
	&0 & 
\end{tikzcd}
\]
The corresponding $\mathbb{B}$-module is the following:
\[
M=(\bigoplus_{i=1}^3 \mathbb{B}x_i) / \angles{x_1+x_2 \sim x_1+x_3 \sim x_2+x_3}
\]

Since $L$ is an atomic lattice, $M$ is quasi-free by Lemma \ref{lemma: atomic}. However, $M$ is not projective (and hence not free). See \cite[Section 7.6]{borger2024facets}.
\end{myeg}

\section{Subgroups of $GL_n$}\label{section: subgroups of $GL_n$}
 
We have a solid understanding of linear actions on free modules, making it natural to extend this study to linear actions on non-free modules.

This requires understanding the automorphism groups of modules. In many cases (e.g. quasi-free modules), the relevant automorphism group comes with a natural embedding into $GL_n$.  In addition, the automorphism group of a tropical linear space (more generally, $K$-linear space) is a subgroup of $GL_n$ by definition. 

We may characterize purely torsion subgroups of $\text{GL}_n(K)$ using our classification of representations of finite groups.  For this, the following lemma will be helpful.

\begin{lem}\label{lemma: invariant basis}
Let $V$ be a free module over an idempotent semifield $K$ and let $G$ be a finite group. Fix a linear $G$-action on $V$. Then $V$ has a basis which is {invariant} under this action.
\end{lem}
\begin{proof}
Let $V=K^n$ and $\{e_i\}$ be the standard basis vectors. 
$G$ acts on the set of basis lines $X = \{\mathrm{span}(e_i)\mid i\in\{1,\ldots,n\}\}$.  By \cite[Proposition 3.15]{JMTmatroidsPart1}, $V$ is determined up to isomorphism by the $G$-set $X$.  We may obtain another representation $W$ whose corresponding $G$-set is $X$ simply by considering the free module with basis $X$.  So $V\cong W$ as modules, and we may prove the result for $W$ instead.  But $X$ is {invariant} under the $G$-action, which implies the result.
\end{proof}

\begin{pro}
Let $K$ be an idempotent semifield.  Let $G\subseteq \text{GL}_n(K)$ be a purely torsion subgroup.  Then $G$ is conjugate to a subgroup of $S_n$.\footnote{We view $S_n$ as a subgroup of $\text{GL}_n(K)$.} In other words, there exist a subgroup $H\subseteq S_n$ and $M\in \text{GL}_n(K)$ such that $G = \{MPM^{-1} \mid P\in H \}$.  
Furthermore, the matrix $M$ can be chosen to be diagonal. 
\end{pro}
\begin{proof}
The inclusion $G\subseteq GL_n(K)$ yields a linear $G$-action on $K^n$; let $V=K^n$ denote this representation, which is clearly faithful.  By Lemma \ref{lemma: invariant basis}, there is some set of basis vectors $x_1,\ldots, x_n$ which is {invariant} under the $G$-action (and note that the invariance does not depend on the ordering of the basis).  By the uniqueness property of bases, $x_1,\ldots, x_n$ agree with $e_1,\ldots,e_n$ up to permutation and rescaling, and without loss of generality we may reorder $x_1,\ldots,x_n$ so that no permutation is needed.  Thus there exist $\lambda_i \in K$ for $1 \leq i \leq n$ such that $\lambda_1e_1,\ldots,\lambda_ne_n$ is invariant under the $G$-action.

Let $\phi: G \rightarrow S_n$ be the restriction of the surjective map $\text{GL}_n(K) \to S_n$ in Proposition \ref{proposition: exact sequence R, GL, S_n}.  Then, for $g \in G$, $g(\lambda_i e_i)$ is a scalar multiple of $e_{\phi(g)(i)}$ and hence
\begin{equation}
g(\lambda_i e_i) = \lambda_{\phi(g)(i)}e_{\phi(g)(i)}.
\end{equation}
Identify $S_n$ with the set of permutation matrices so that we have $\phi(g)_{ji}=\delta_{j,\phi(g)(i)}$. Let $\Lambda$ be the diagonal matrix with diagonal entries $\lambda_1,\ldots,\lambda_n$, we may write this as
\begin{equation}
g (e_i) = \sum_{j} \frac{1}{\lambda_i}\phi(g)_{ji} \lambda_{j} e_j,
\end{equation}
which in matrix form is simply $g = \Lambda \phi(g)\Lambda^{-1}$. With $M=\Lambda$ and $H = \mathrm{im} \phi$ the result follows. 
\end{proof}

In order to study more general subgroups of $\text{GL}_n(K)$ where $K$ is a zero-sum-free semifield, it is natural to decompose such a subgroup in a manner analogous to the decomposition of $\text{GL}_n(K)$ into permutations and diagonal matrices.

\begin{lem}\label{lemma: Decomposition of quasi-free automorphisms}
Let $K$ be a zero-sum-free semifield and $G$ be a subgroup of $\text{GL}_n(K)$. Let $\pi: GL_n(K) \rightarrow S_n$ be the canonical homomorphism in Proposition \ref{proposition: exact sequence R, GL, S_n}.  Then the sequence
\begin{equation}
1 \rightarrow G \cap (K^\times)^n \rightarrow G \rightarrow \pi(G) \rightarrow 1
\end{equation}
is exact.
\end{lem}
\begin{proof}
This follows trivially from Proposition \ref{proposition: exact sequence R, GL, S_n}.
\end{proof}

The above lemma decomposes a subgroup of $GL_n(K)$ into a subgroup of the torus $(K^\times)^n$ and a subgroup of the permutation group $S_n$.  Thus it is natural to focus on classifying such subgroups.  For the moment, we will focus on studying subgroups of the torus under some additional assumptions.

\begin{mydef}\label{definition: linear subgroup}

Let $K$ be a semifield. A subgroup $G\subseteq \text{GL}_n(K)$ is called a \emph{linear subgroup} of $\text{GL}_n(K)$ if $G$ is the set of matrices in $\text{GL}_n(K)$ such that there is a system of linear equations
\begin{equation}\label{eq: linear subgroup}
\sum_{i,j} A_{ik} x_{ij}= \sum_{i,j} A_{ij} y_{ij}  
\end{equation}
for some scalars $x_{ij}$ and $y_{ij}$, and all $A = (a_{ij})$ and $A^{-1}=(b_{ij})$ in $G$ satisfy the system of linear equations, i.e., 
\[
\sum_{i,j} a_{ij} x_{ij}= \sum_{i,j} a_{ij} y_{ij} \quad \textrm{ and } \quad \sum_{i,j} b_{ij} x_{ij}= \sum_{i,j} b_{ij} y_{ij}
\]

Such linear equations are called a \emph{defining system of equations} of $G$. 
\end{mydef}

Similarly, a subgroup $G\subseteq \text{GL}_n(K)$ is called a \emph{linear subgroup} of $(K^\times)^n$ if $G$ is the the set of $d = (d_1, \dots, d_n)\in (K^\times)^n$ (viewed as invertible diagonal matrices $D$ in $\text{GL}_n(K)$ with diagonal entries $d$) such that there is a system of linear equations $\sum_i d_{i} x_i= \sum_j d_{j} x_j$, and all $D$ and $D^{-1}$ in $G$ satisfy the system of linear equations. Such linear equations are called a \emph{defining system of equations} of $G$.

\begin{rmk}\label{rmk: linear subgroup}
\begin{enumerate}
    \item 
It is clear that if $G$ is a linear subgroup of $\text{GL}_n(K)$, where $K$ is a semifield then $G\cap (K^\times)^n$ is a linear subgroup of $(K^\times)^n$, where we think of $(K^\times)^n$ as the group of invertible diagonal matrices.
\item 
One can easily check that when $K$ is a semifield $\text{GL}_n(K)$ itself is a linear subgroup of itself since with $x_i=e_i$, any $A \in \text{GL}_n(K)$ satisfies \eqref{eq: linear subgroup}.
\end{enumerate}

\end{rmk}

For many interesting subgroups $G\subseteq \text{GL}_n(K)$, it turns out that $G\cap (K^\times)^n$ is the intersection of $(K^\times)^n$ with the solution set of a family of linear equations.

\begin{lem}\label{lemma: Automorphism groups of weakly free modules are linear subgroups}\label{lemma: linear group quasi}
Let $K$ be a zero-sum-free semifield. Let $M$ be a weakly free or quasi-free module of rank $n$ which can be embedded into a (possibly infinite-dimensional) free module.  Then $\mathrm{Aut}(M)$ is a linear subgroup of $GL_n(K)$.  Furthermore, if $M$ is a quotient of the free module spanned by its quasi-basis by finitely many relations and $M$ is contained in a finitely generated free module, then there is a finite defining system of equations for $G$.
\end{lem}
\begin{proof}
Fix a weak basis (resp. quasi-basis) $x_1,\ldots,x_n$ for $M$. Let $f:M \to M$ be an automorphism. Then $f(x_1),\ldots,f(x_n)$ form a weak basis (resp. quasi-basis) of $M$. It follows from Lemma \ref{lemma: quasi-basis permute} or Lemma \ref{lemma: embedding of automorphisms into GL_n} that there exists a unique $A \in \text{GL}_n(K)$ such that
\[
f(x_i)=\sum_j A_{ij}x_j.
\]
One can easily see that $\iota:\text{Aut}(M) \to \text{GL}_n(K)$ sending $f$ to $A$ is an injective group homomorphism.\footnote{Note that an embedding $\iota$ depends on a choice of a weak basis or quasi-basis.} 

If there is no nontrivial relation among $x_1,\dots,x_n$, then $M=K^n$. In particular, $\text{Aut}(M)=\text{GL}_n(K)$, and hence $\text{Aut}(M)$ is linear subgroup (see Remark \ref{rmk: linear subgroup}).

Now, let $A \in \text{GL}_n(K)$. Suppose $x_1,\ldots, x_n$ satisfy a linear relation for some $a_i$, $b_i \in K$
\begin{equation}\label{eq: nontrivial rel}
\sum_i a_i x_i = \sum_i b_i x_i.
\end{equation}
If $f\in \mathrm{Aut}(M)$ corresponds to $A$ under $\iota$, then we have
\begin{equation}
\sum_i a_if(x_i) = \sum_i b_if(x_i).
\end{equation}
Equivalently, we have
\begin{equation}\label{equation: linear relations for automorphisms of weakly free modules}
 \sum_{i,j} a_i A_{ij} x_j = \sum_{i,j} b_i A_{ij} x_j.   
\end{equation}
Now view $M$ as a subset of a free module, so that each side of \eqref{equation: linear relations for automorphisms of weakly free modules} may be viewed as a vector.  Then by focusing on one coordinate of the above vector-valued equation, we get a scalar-valued linear equation that the $A_{ij}$ will satisfy if $A\in \mathrm{Aut}(M)$. In particular, if we let
\[
X:=\{A \in \text{GL}_n(K) \mid \textrm{ $A$ and $A^{-1}$ satisfies \eqref{equation: linear relations for automorphisms of weakly free modules} for each relation \eqref {eq: nontrivial rel}}\},
\]
then $\text{Aut}(M) \subseteq X$ via the injection $\iota$. We only have to prove the other inclusion to show that $\text{Aut}(M)$ is a linear subgroup. 

Conversely, suppose $A \in X$. Since $M$ is the quotient of the free module spanned by $x_1,\ldots,x_n$ by the relations \eqref{eq: nontrivial rel}, and since the variables $\sum_j A_{ij} x_j$ satisfy these relations, there is a homomorphism $\phi: M\rightarrow M$ with $\phi(x_i) = \sum A_{ij} x_j$.  Similarly $A^{-1}$ satisfies these equations if and only if there is a homomorphism $\psi: M\rightarrow M$ with $\psi(x_i) = \sum (A^{-1})_{ij} x_j$.

Therefore, $A$ and $A^{-1}$ both satisfy the system of equations if and only if there exist $\phi, \psi$ as above.   If this holds, then clearly $\phi(\psi(x_i)) = x_i = \psi(\phi(x_i))$ for all $i$. Moreover, since the weak basis (resp. quasi-basis) is a generating set, this implies $\phi$ is invertible.  Furthermore, it is clear that $\phi$ maps to $A\in GL_n(K)$ under our embedding $\iota: \mathrm{Aut}(M)\to GL_n(K)$. Therefore, we have $X \subseteq \text{Aut}(M)$, as desired.  

It is clear from the construction that the system of equations is finite if $M$ is determined by finitely many relations on the weak basis or quasi-basis and if the containing free module is finite-dimensional.
\end{proof}

\begin{lem}\label{lemma: Automorphism groups of tropical spaces are linear subgroups}
Let $K$ be an idempotent semifield. Let $V\subseteq K^n$ be a $K$-linear space (Definition \ref{definition: K-linear space}), and let $\mathrm{Aut}(V)$ denote the automorphisms $\phi$ of $K^n$ such that $\phi$ and $\phi^{-1}$ preserve $V$.  Then $\mathrm{Aut}(V)$ is a linear subgroup of $GL_n(K)$.  The defining system of equations can be taken to be finite if $V$ is finitely generated and is the intersection of finitely many $K$-hyperplanes. 

\end{lem}
\begin{proof}
By definition, $V$ is the set of solutions in $K^n$ of a family of linear relations $\mathcal{B}$.  An automorphism of $V$ is an element $A\in GL_n(K)$ such that $A$ and $A^{-1}$ preserve $V\subseteq K^n$. Equivalently, for each linear relation in $\mathcal{B}$, 
\begin{equation}\label{eq:bend}
\sum \alpha_i x_i = \sum \beta_i x_i
\end{equation}
and for each $x\in V$, the linear relation must also be satisfied by $Ax$ and $A^{-1}x$.  That is
\begin{equation}\label{eq: tropical linear equation for Ax}
\sum_{i, j} \alpha_i A_{ij} x_j = \sum \beta_i A_{ij} x_j
\end{equation}
and 
\begin{equation}
\sum_{i, j} \alpha_i (A^{-1})_{ij} x_j = \sum \beta_i (A^{-1})_{ij} x_j.
\end{equation}
As in Lemma \ref{lemma: linear group quasi}, one can easily see that this shows $\text{Aut}(V)$ is a linear subgroup.
Furthermore, we remark that we only need to impose these equations for $x$ in a set of generators rather than for all $x\in V$, since if $A: K^n \rightarrow K^n$ maps the generators into $V$ then it maps all of $V$ into $V$.

Finally, it is clear from construction that the system of equations is finite if $V$ is finitely generated and is the intersection of finitely many $K$-hyperplanes.
\end{proof}

Our next task will be to classify linear subgroups of $(K^\times)^n$ in the case $K = \T$.  Note that we may make the identification $(K^\times)^n = \mathbb{R}^n$.  We will make use of the following lemma to reduce to the case where there is a vector in the linear subgroup whose components are all distinct.

\begin{lem}\label{lemma: index lemma}
Let $A$ be a torsion-free abelian group.  Let $B\subseteq A^n$ be a subgroup.  
\begin{enumerate}
    \item Suppose that for every $b\in B$ there is some index $i$ such that $b_i = 0$.  Then there is some $i$ such that $b_i = 0$ for all $b\in B$; that is, the choice of $i$ doesn't depend on $b$.
    \item Suppose that for every $b\in B$ there are indices $i, j$ such that $b_i = b_j$.  Then there exist $i, j$ such that $b_i = b_j$ for all $b\in B$.

\end{enumerate}
\end{lem}
\begin{proof}Both parts are proven in the same way, but we will focus on the second part as it is the only one we will use.  Since $A$ is torsion free, it embeds into $A\otimes_{\mathbb{Z}} \mathbb{Q}$, so we may speak of $\mathbb{Q}$-linear combinations of elements of $A$ or even of elements of $B$.  We will prove the contrapositive of the lemma, so we assume that for every pair $i\neq j$ there exists some $w\in B$ with $w_i\neq w_j$.

Let $v\in B$ be an element with the largest possible number of pairs $i, j$ with $v_i\neq v_j$.  If all entries are distinct, then the lemma holds; otherwise pick $k, l$ such that $v_k = v_l$.  Let $w$ be such that $w_k \neq w_l$.  Now we consider $\mathbb{Q}$-linear combinations (which may not lie in $B$) of the form
\begin{equation}
x = \alpha v + \beta w
\end{equation}
If $i, j$ are distinct indices such that either $v_i \neq v_j$ or $w_i \neq w_j$, then $x_i = x_j$ would not hold for all choices of $\alpha$ and $\beta$, so would only hold for $(\alpha, \beta)$ lying in a subspace of dimension at most $1$ in $\mathbb{Q}^2$.  Since $\mathbb{Q}^2$ is not the union of a finite collection of lines, we may choose $\alpha$ and $\beta$ such that $x_i \neq x_j$ whenever $v_i \neq v_j$ or $w_i \neq w_j$.  But now $x$ has more distinct pairs of entries than $v$; it has a distinct pair wherever $v$ does, but also $x_k\neq x_l$ since $w_k\neq w_l$.  Finally, we clear denominators - if $n$ is a common denominator for $\alpha, \beta$, then $nx$ has distinct pairs of entries in the same slots as $x$, but $nx = (n\alpha)v + (n\beta)w$ lies in $B$.  This contradicts the maximality of $v$.
\end{proof}

In the tropical setting, the following lemma is significant step towards classifying subgroups of the torus which are given by linear equations.
\begin{lem}\label{lemma: main lemma}
Consider a system of tropical linear equations of the form
\begin{equation}\label{equation: tropical-linear equation}
\max(a_1 + x_1, \ldots, a_n + x_n) = \max(b_1 + x_1, \ldots, b_n + x_n), 
\end{equation}
where the coefficients $a_1,\ldots,a_n,b_1,\ldots,b_n$ are in $\T$.\footnote{In \eqref{equation: tropical-linear equation}, addition is the usual addition of real numbers corresponding to the tropical multiplication.}  Let $G\subseteq\mathbb{R}^n$ be the set of $x$ such that $x$ and $(-x)$ both satisfy the above system of equations.  Suppose $G$ is a subgroup.  Then
\begin{enumerate}
    \item 
Suppose $v=(v_1,\dots,v_n)\in G$ and $k\in \{1,\ldots, n\}$ are such that $v_k > v_i$ for each $i \neq k$.  Then each of the equations defining $G$ has $a_k = b_k$. 
    \item 
Suppose the system of equations is finite.  Suppose also that $v=(v_1,\dots,v_n)\in G$ and $k$ are such that $v_k > v_i$ for each $i \neq k$.  Then $te_k\in G$ for all $t\in\mathbb{R}$ ($e_k$ is the standard basis vector).
    \item Suppose the system of equations is finite and that $G$ contains an element whose coordinates are distinct.  Then $G = \mathbb{R}^n$.
\end{enumerate}
\end{lem}
\begin{proof}For the first part, fix one of the equations defining $G$ and write it as
\begin{equation}\label{equation: tropical-linear equation2}
\max(a_1 + x_1, \ldots, a_n + x_n) = \max(b_1 + x_1, \ldots, b_n + x_n).
\end{equation}
If $a_k \neq -\infty$, then by picking $M \in \N$ large enough, we can ensure $a_k + Mv_k > a_i + Mv_i$ for $i\neq k$. By using that $Mv\in G$ (and hence $Mv$ satisfies the defining equations), we have
\begin{equation}
a_k + Mv_k = \max(a_1 + Mv_1, \ldots, a_n + Mv_n) = \max(b_1 + Mv_1, \ldots, b_n + Mv_n).
\end{equation}

There is some $l$ such that the $l$th term is the maximum on the right side for infinitely many such $M$, and hence
\begin{equation}
    a_k + Mv_k = b_l + Mv_l
\end{equation}
for infinitely many choices of $M$.  Since the left side is finite, so is the right.  And by taking the difference between this equation for two values of $M$ we get $v_k = v_l$.  Since $v_k > v_i$ for $i \neq k$,  we have $k = l$.  Subtracting $Mv_k$ gives $a_k = b_k$ as desired.  The same argument works if $b_k\neq -\infty$, and the claim is trivial if both are $-\infty$.

For the second part, we again begin by fixing one of the defining equations.  We claim that for all sufficiently large $M$, there exists $\delta > 0$ such that $Mv + te_k$ and its negative satisfy the selected equation whenever $|t| < \delta$. 

We consider two cases according to whether $a_k, b_k$ are finite.  Note that if $a_k = b_k = -\infty$, then $x_k$ has no influence on \eqref{equation: tropical-linear equation2}, and so we may freely add or subtract any $t\in \mathbb{R}$ from the $k$th component.  Since $v\in G$, $Mv \in G$ for any $M\in\mathbb{N}$ and so $Mv$ and $-Mv$ satisfy the tropical linear equation and hence so do $Mv + te_k$ and $-Mv - te_k$ for any $t$.

Now, let's assume $a_k$ and $b_k$ are finite.  Since $v_k > v_i$ for each $i\neq k$, if $M\in\mathbb{N}$ is large enough then $a_k + Mv_k > a_i + Mv_i$ for all $i\neq k$, and similarly for $b_k$.  By the same argument, if we choose $M$ large enough, then we also have $a_k - Mv_k < a_i - Mv_i$ for $i\neq k$ and similarly for $b_k$.

If $\delta > 0$ is small enough, then in each of these strict inequalities, the two sides differ by more than $\delta$.  Thus $a_k + Mv_k + t > a_i + Mv_i$, $a_k - Mv_k - t < a_i - Mv_i$, $b_k + Mv_k + t > b_i + Mv_i$ and $b_k - Mv_k - t < b_i - Mv_i$ for all $|t| < \delta$.  Note that $Mv\in G$, and so $Mv$ satisfies \eqref{equation: tropical-linear equation2}, i.e.
\begin{equation}
\max(a_1 + Mv_1, \ldots, a_n + Mv_n) = \max(b_1 + Mv_1, \ldots, b_n + Mv_n).
\end{equation}
Suppose $|t| < \delta$, and observe that,
\begin{equation} 
\max(a_1 + Mv_1, \ldots, a_k + Mv_k + t, \ldots, a_n + Mv_n) = a_k + Mv_k + t
\end{equation}
and using $a_k = b_k$ this is
\begin{equation}
b_k + Mv_k + t = \max(b_1 + Mv_1, \ldots, b_k + Mv_k + t, \ldots, b_n + Mv_n)
\end{equation}
so $Mv + te_k$ satisfies \eqref{equation: tropical-linear equation2}.  

On the other hand $-Mv$ must also satisfy \eqref{equation: tropical-linear equation2}, so
\begin{equation}\label{equation: tropical inverse linear equation}
\max(a_1 - Mv_1, \ldots, a_n - Mv_n) = \max(b_1 - Mv_1, \ldots, b_n - Mv_n).
\end{equation}
For $|t| < \delta$, $a_k - Mv_k - t < a_i - Mv_i$ implies the $k$th term is not the maximum in either the left side of the previous equation or in
\begin{equation}
\max(a_1 - Mv_1,\ldots, a_k - Mv_k - t,\ldots, a_n - Mv_n)
\end{equation}
and so we may drop the $k$th term in either expression, and furthermore since only the $k$th term differs, both expressions are equal, i.e.
\begin{equation}
\max(a_1 - Mv_1,\ldots, a_k - Mv_k - t,\ldots, a_n - Mv_n) = \max(a_1 - Mv_1, \ldots, a_n - Mv_n).
\end{equation}
Combining this (and the analogue for $b_k$) with \eqref{equation: tropical inverse linear equation},
\begin{equation}
\begin{aligned}
    \max(a_1 - Mv_1, \ldots, a_k - Mv_k - t,\ldots, a_n - Mv_n)\\
  = \max(b_1 - Mv_1, \ldots, b_k - Mv_k - t, \ldots, b_n - Mv_n).   
\end{aligned}
\end{equation}

In both cases, we have seen that $Mv + te_k$ and $-Mv - te_k$ satisfy the selected tropical linear equation for all $|t|<\delta$ so long as $M$ is large enough and $\delta$ is small enough (with the threshold depending on $M$).  Since there are finitely many equations defining $G$, we first select $M$ large enough for all of them and then select $\delta$ small enough for all of them.  Then $Mv + te_k$ and its negative satisfy all of the defining equations and hence lie in $G$ whenever $|t|<\delta$.  Hence so does $(Mv + te_k) - Mv = te_k$.  Furthermore for any real $s$, $se_k$ is an integer multiple of $te_k$ for some small $t$, and hence lies in $G$.

For the third part, let $I$ be the set of indices $j$ such that $te_j\in G$ for all $G$.  Suppose for the sake of contradiction that $I\neq \{1,\ldots,n\}$.  Let $v\in  G$ have distinct coordinates.  For each $j\in I$, we may change the $j$th coordinate of $v$ arbitrarily without leaving $G$ (because $v + te_j\in G$).  After such modifications, we can produce a vector $w\in G$ with distinct coordinates such that the largest coordinate is not the $j$th for any $j\in I$.  Let $k$ be the index of the largest coordinate.  Then by the previous part, $te_k\in G$ for all $t$ and hence $k\in I$, a contradiction.  Thus $te_i\in G$ for all $i$.  This clearly implies $G = \mathbb{R}^n$.
\end{proof}

\begin{pro}\label{proposition: tropical index}
Let $G\subseteq (\T^\times)^n = \mathbb{R}^n$ be a linear subgroup with a finite defining system of equations.  Then there is some equivalence relation $\sim$ on $\{1,\ldots,n\}$ such that
\[
G = \{x\in\mathbb{R}^n\mid x_i = x_j \;\mathrm{if}\; i\sim j\}.
\]
\end{pro}
\begin{proof}
We proceed by induction on dimension, with the $0$-dimensional case being trivial. From Lemma \ref{lemma: main lemma} we know that if $G$ contains a vector with distinct coordinates, then $G = \mathbb{R}^n$.  So we may assume that for every $v\in G$ there exist $i, j$ such that $v_i = v_j$. Then, from Lemma \ref{lemma: index lemma}, there exists a pair $i, j$ such that for all $v\in G$ we have $v_i = v_j$.  

We may then add the tropical linear equation $x_i = x_j$ to the set of defining equations for $G$.  For each other defining equation (say $\displaystyle\max_k \{a_k + x_k\} = \displaystyle\max_k \{b_k + x_k\}$), we may replace $x_i$ with $x_j$ to obtain
\begin{equation}\label{eq: Tropical linear with two combined terms}
\max(\max_{k\not\in\{i,j\}} \{a_k + x_k\}, \max\{a_i, a_j\} +x_j) = \max(\max_{k\not\in\{i,j\}} \{b_k + x_k\}, \max\{b_i, b_j\} +x_j)
\end{equation}

Note that the map $\phi: G\rightarrow \mathbb{R}^{n-1}$ obtained by deleting the $i$th coordinate is injective (as $\phi(g)$ still determines the $j$th coordinate).  Let $H = \{ (x_1,\ldots, \widehat{x_i},\ldots, x_n) \mid x\in G \}$ be its image.  Clearly $H$ is a subgroup of $\mathbb{R}^{n-1}$.  It is also clear that $H$ is the set of $x\in \mathbb{R}^{n-1}$ such that $x$ and $(-x)$ satisfy each of the equations \eqref{eq: Tropical linear with two combined terms}.  Therefore, by the inductive hypothesis, there is some equivalence relation $\equiv$ on $\{1,\ldots,n\} \backslash \{i\}$ such that 
\[
H = \{ x \in \mathbb{R}^{n-1} \mid x_k = x_l \;\mathrm{if}\; k\equiv l\}.
\]

Now we define $\sim$ so that $k \sim l$ if either $k, l \neq i$ and $k \equiv l$ or if $k, l \in \{i, j\}$. Then, one can easily see that $\sim$ satisfies the assertion on $\{1,\dots,n\}$.
\end{proof}

There is a similar result for the semifield $\mathbb{R}_{\geq 0}$.

\begin{pro}\label{proposition: nonnegative index}
Let $G\subseteq (\mathbb{R}_{\geq 0}^\times)^n $ be a linear subgroup. Then there is some equivalence relation $\sim$ on $\{1,\ldots,n\}$ such that $G = \{x\in(\mathbb{R}_{\geq 0}^\times)^n\mid x_i = x_j \;\mathrm{if}\; i\sim j\}$.
\end{pro}
\begin{proof}
We proceed by induction on $n$ with the $n = 0$ case being trivial.  We consider two cases according to whether there is an element of $G$ with distinct coordinates.

First suppose there exists $w\in G\subseteq (\mathbb{R}_{\geq 0}^\times)^n$ such that the $w_i$ are all distinct.  For all integers $k$, we have $(w_1^k, \ldots, w_n^k)\in G$, so all such elements satisfy the defining linear equations of $G$.  On the other hand, vectors of the form $(w_1^k, \ldots, w_n^k)$ span $\mathbb{R}^n$ by the nonvanishing of the Vandermonde determinant.  So all of $\mathbb{R}^n$ satisfies the defining linear equations of $G$, and in particular this is true for all elements of $(\mathbb{R}_{\geq 0}^\times)^n$.  Hence $G = (\mathbb{R}_{\geq 0}^\times)^n$.

In the other case, where $G$ does not contain an element with distinct entries, as in Proposition \ref{proposition: tropical index} there exists $i, j$ such that $v_i = v_j$ for all $v\in G$.  As in the proof of Proposition \ref{proposition: tropical index}, we can use the inductive hypothesis on $\{ (v_1,\ldots, v_{i-1}, v_{i+1},\ldots, v_n)\mid v\in G\}$. 
\end{proof}

\section{Automorphisms of weakly free modules and $K$-linear spaces}

In this section, we study automorphism groups of weakly free modules and $K$-linear spaces, and prove one of the main theorems of this work. The result is used in the subsequent section on valuated matroidal representations. To this end, we recall some results from group cohomology.

First, recall that there is a well-known relation between group extensions and degree 2 group cohomology. Consider the exact sequence of groups (considered with multiplication)
\begin{equation}
1 \rightarrow A \rightarrow B \rightarrow C \rightarrow 1
\end{equation}
with $A$ abelian. In such a setting, the action of $B$ on $A$ by conjugation descends to an action of $C$.  We call $B$ an \emph{extension} of $C$ by the $C$-module $A$.  Two such extensions $B, B'$ are considered isomorphic if there is a group isomorphism $B\rightarrow B'$ compatible with the inclusion maps from $A$ and the projection maps to $C$.  We recall the following standard result (see, for instance, \cite[Theorem 6.6.3]{weibel1994introduction}).

\begin{pro}\label{proposition: H^2 and extensions}
Let $C$ be a group and $A$ be an $C$-module.  There is a one-to-one correspondence between extensions of $C$ by $A$ and elements of $H^2(C, A)$.  Moreover the zero element $0\in H^2(C, A)$ corresponds to the semidirect product $C \ltimes A$. 
\end{pro}

The following lemma will be used to find cohomological obstructions to various statements about group actions on modules over connected zero-sum-free semirings.  While we already understand actions on free modules, a motivating example of the use of Lemma \ref{lemma: cohomological classification of homomorphisms} is to apply it to
\begin{equation}\label{eq: ex seq}
1 \rightarrow (K^\times)\rightarrow GL_n(K)\rightarrow S_n \rightarrow 1
\end{equation}
in order to understand the extent to which the action on a free module is determined by the action on the set of basis lines.

Fix an exact sequence of groups
\begin{equation}
1 \rightarrow A \rightarrow B \xrightarrow{\pi} C \rightarrow 1
\end{equation}
with $A$ abelian.  Let $G$ be another group.  Call two homomorphisms $\phi, \psi: G \rightarrow B$ \textit{equivalent} if there is some $a\in A$ such that 
\[
\phi(g)=a\psi(g)a^{-1}, \quad \forall g \in G.
\]
Fix a homomorphism $\phi: G\rightarrow B$. Then the following holds. 

\begin{lem}\label{lemma: cohomological classification of homomorphisms}   
The set of equivalence classes of homomorphisms $\psi:G\rightarrow B$ with $\pi\psi = \pi\phi$ is in one-to-one correspondence with the elements of $H^1(G, A)$, where $A$ is a $B$-module via conjugation and hence a $G$-module via $\phi$.
\end{lem}
\begin{proof}
As a first step, suppose $\psi:G \rightarrow B$ is a homomorphism such that $\pi\psi = \pi\phi$.  Define
\[
\eta: G\rightarrow A, \quad g \mapsto \psi(g)\phi(g)^{-1}.
\]
Note that $\psi(g)\phi(g)^{-1}$ lies in $A$ for any $g \in G$ since $\pi\psi = \pi\phi$. Now, observe that for any $g_1,g_2 \in G$, we have
\begin{equation}
\eta(g_1g_2) =  
\psi(g_1) \eta(g_2)\phi(g_1)^{-1} = \eta(g_1) (\phi(g_1) \eta(g_2) \phi(g_1)^{-1}).
\end{equation}
The parenthesized expression is the result of $g_1$ acting on $\eta(g_2)$, from which we see that $\eta: G\rightarrow A$ is a 1-cocycle.  

Conversely, given an 1-cocycle $\eta$, let
\[
\psi: G \to B, \quad g \mapsto \eta(g)\phi(g),
\]
and observe that for any $g_1,g_2 \in G$, we have
\begin{equation}
\psi(g_1g_2) = \eta(g_1g_2)\phi(g_1g_2) = \eta(g_1)(\phi(g_1) \eta(g_2) \phi(g_1)^{-1})(\phi(g_1)\phi(g_2)) = \psi(g_1)\psi(g_2)
\end{equation}
Thus homomorphisms $G\rightarrow B$ lying above $\pi\phi: G\rightarrow C$ are in one-to-one correspondence with 1-cocycles. 

Next, suppose $\psi, \psi':G \rightarrow B$ are equivalent homomorphisms such that $\pi\phi = \pi\psi = \pi \psi'$.  Then there is some $a\in A$ such that $\psi' = a\psi a^{-1}$.  Define $\eta(g) = \psi(g)\phi(g)^{-1}$ and similarly for $\eta'$, which are $1$-cocycles from above. The difference between the two cocycles is
\begin{equation}\label{eq: difference}
\eta(g)\eta'(g)^{-1} = \psi(g)\phi(g)^{-1}\phi(g)\psi'(g)^{-1} = \psi(g) \psi'(g)^{-1} = \psi(g) a\psi(g)^{-1} a^{-1}.
\end{equation}
Using $\psi(g) = \eta(g)\phi(g)$ and that $\eta(g)\in A$ commutes with any other element of $A$, {in particular with the element $\phi(g) a \phi(g)^{-1}$,} the above expression becomes the following
\begin{equation}
\eta(g)(\phi(g) a \phi(g)^{-1}) \eta(g)^{-1} a^{-1} =  \eta(g)\eta(g)^{-1}\phi(g) a \phi(g)^{-1}  a^{-1} = \phi(g) a \phi(g)^{-1}  a^{-1}.
\end{equation}
This is a 1-coboundary.

Conversely, suppose $\psi, \psi'$ correspond to cocycles that differ by a coboundary.  The above computation shows the difference is $\psi(g) \psi'(g)^{-1}$.  So there is some $a\in A$ such that
\[
\psi(g)\psi'(g)^{-1} = \phi(g) a \phi(g)^{-1}  a^{-1}.
\]
Rearranging and using commutativity of $A$ gives
\[
\psi(g) = \phi(g) a \phi(g)^{-1}  a^{-1} \psi'(g) = (\phi(g) a \phi(g)^{-1})  (a^{-1} \eta'(g))\phi(g)
\]
\[
= a^{-1} \eta'(g)\phi(g) a \phi(g)^{-1} \phi(g)=a^{-1} \eta'(g)\phi(g) a = a^{-1}\psi'(g) a.
\]
So $\psi$ and $\psi'$ are equivalent as desired.
\end{proof}

\begin{myeg}
Fix a short exact sequence $0\rightarrow A \rightarrow B \rightarrow B/A \rightarrow 0$ of abelian groups.  Fix a (possibly non-abelian) group $G$ and a homomorphism $\phi: G\rightarrow B$.  By commutativity, $\phi(g) = a\psi(g)a^{-1}$ simplifies to $\phi(g) = \psi(g)$, so equivalence of homomorphisms in the above sense is just equality.  So homomorphisms that are congruent to $\phi$ modulo $A$ are classified by $H^1(G, A)$.  Note that $G$ acts on $A$ trivially because $A$ is contained in the center of $B$.  So $H^1(G, A) = \Hom(G, A)$.  In fact it is easy to directly verify the proposition in this case: If $\psi: G\rightarrow B$ is congruent to $\phi$ modulo $A$, then $\phi \psi^{-1}: G\rightarrow A$ is a homomorphism. 
\end{myeg}

\begin{lem}\label{lemma: basis extension lemma}
Let $M$ be a quasi-free module over a zero-sum-free semifield $K$.  Let $x_1,\ldots, x_n\in M$ be a quasi-basis.  Then the induced elements $\bar{x}_1,\ldots,\bar{x}_n$ form a quasi-basis for $M\otimes_K \mathbb{B}$ as a $\mathbb{B}$-module.  In particular, $\bar{x}_1,\ldots,\bar{x}_n$ are all distinct.
\end{lem}
\begin{proof}

First, we note that for any zero-sum-free semifield $K$, there is a unique map $K \to \mathbb{B}$ sending $a \neq 0$ to $1$ and $0$ to $0$. In particular, $\mathbb{B}$ is a $K$-module. Now, by Proposition \ref{proposition: characterization of quasi-free modules}, there is some congruence relation $\sim$ such that $K^n / \sim$ is isomorphic to $M$. Now, $M\otimes_K \mathbb{B}$ is the quotient of $M$ by relations of the form $\lambda x = x$ for $\lambda\neq 0$.  We may also arrive at $M\otimes_K \mathbb{B}$ instead by first quotienting $K^n$ by $\lambda x = x$ for $\lambda\neq 0$ to get $\mathbb{B}^n$ and then quotienting by the relations that generate $\sim$. As the relations in $\sim$ can only identify pairs of vectors that are equal or that both have at least two nonzero coordinates, the equivalence classes of the standard basis form a quasi-basis for this quotient of $\mathbb{B}^n$ by Proposition \ref{proposition: quasi-free quotient}.  In particular, these equivalence classes are distinct.  By inspection these equivalence classes are the image of the quasi-basis of $M$.  So no two quasi-basis elements of $M$ map to the same element of $M\otimes_K \mathbb{B}$.
\end{proof}

Let $M$ be a quasi-free module over a zero-sum-free semifield $K$.  We may view $\mathrm{Aut}(M)$ as a subgroup of $\text{GL}_n(K)$ by Lemma \ref{lemma: embedding of automorphisms into GL_n}, and as $M\otimes_K \mathbb{B}$ is quasi-free by Lemma \ref{lemma: basis extension lemma}, we may view $\mathrm{Aut}(M\otimes_K\mathbb{B})$ as a subgroup of $S_n = \text{GL}_n(\mathbb{B})$. 

Note that there is a canonical homomorphism $\phi:\mathrm{Aut}(M) \rightarrow \mathrm{Aut}(M\otimes_K \mathbb{B})$.  So the study of $\mathrm{Aut}(M)$ reduces to studying the kernel and image.  In the case when $M$ is a quasi-free module, we can characterize them as follows. 

\begin{lem}\label{lemma: Base change of quasi-free module to B version 2} 
With the same notation as above, the following hold.
\begin{enumerate}
    \item 
$\ker \phi = \mathrm{Aut}(M) \cap (K^\times)^n$, and 
\item 
the image of $\phi$ equals the image of $\mathrm{Aut}(M)$ under the canonical map $\pi: \text{GL}_n(K)\rightarrow S_n$.
\end{enumerate}
\end{lem}
\begin{proof}
(1): Let $x_1,\ldots,x_n$ be a quasi-basis of $M$.  Let $q: M\rightarrow M\otimes_K\mathbb{B}$ be the quotient map described in Lemma \ref{lemma: basis extension lemma}.  

Let $f\in \ker \phi$.  Then $q(x) = q(f(x))$ for all $x\in M$.  By viewing $f$ as an element of $\text{GL}_n(K)$ and utilizing the characterization of $\text{GL}_n(K)$ in the connected zero-sum-free setting, we see that for any $i$ there exist $\lambda_i\in K^\times$ and $\sigma(i)$ such that $f(x_i) = \lambda_i x_{\sigma(i)}$.  Then
\[
q(x_i) = q(\lambda_i x_{\sigma(i)}) = q(x_{\sigma(i)}).
\]

But, from Lemma \ref{lemma: basis extension lemma}, this implies $i = \sigma(i)$ for all $i$.  So $f(x_i) = \lambda_i x_i$ for all $i$, and $f$ is diagonal.

Conversely let $f\in \mathrm{Aut}(M)$ be diagonal, say $f(x_i) = \lambda_i x_i$ for each $i$.  Then
\[
q(f(x_i)) = q(\lambda_i x_i) = q(x_i),
\]
implying $q\circ f =q$. In particular, $f \in \ker \phi$.

(2): Let $i: \mathrm{Aut}(M)\rightarrow \text{GL}_n(K)$ and $j: \mathrm{Aut}(M\otimes_K \mathbb{B})\rightarrow S_n$ be the inclusion maps.  Then the claim reduces to showing $j\phi = \pi i$.

Fix a quasi-basis $x_1,\ldots, x_n$ for $M$.  Let $f\in \mathrm{Aut}(M)$.  Let $A = i(f)\in \text{GL}_n(K)$.  This satisfies
\begin{equation}f(x_i) = \sum_j A_{ji} x_j.
\end{equation}
We will use bars over elements of $M$ or $K$ to denote the corresponding elements of $M\otimes_K \mathbb{B}$ or $\mathbb{B}$. We will also use bars over elements of $\text{GL}_n(K)$ to denote the corresponding element of $\text{GL}_n(\mathbb{B})$. Then we have
\begin{equation}\overline{f(x_i)} = \sum_j \bar{A}_{ji} \bar{x}_j.
\end{equation}
Note that $\phi(f)$ is the map sending $\bar{x}_i$ to $\overline{f(x_i)}$, and so the above equation says that 
\[
j(\phi(f)) = \bar{A}\in \text{GL}_n(\mathbb{B}) = S_n.
\]
On the other hand, $A = i(f)$ implies $\bar{A} = \pi(i(f))$.  So $j\phi = \pi i$ as desired.
\end{proof}

In the case of finite groups, we have the following well-known vanishing theorem, which will be used to describe $\text{Aut}(M)$. For instance, see \cite[III.10.2]{brown2012cohomology}.

\begin{pro}\label{proposition: cohomology of representation}Let $G$ be a finite group.  Let $V$ be a (possibly infinite dimensional) representation of $G$ over a field of characteristic zero.  Then $H^k(G, V) = 0$ for $k \geq 1$.
\end{pro}

A special class of subspaces of $\mathbb{R}^n$ will come up frequently, and hence we define them as follows. 

\begin{mydef}\label{definition: partition space}
Let $V\subseteq \mathbb{R}^n$ be a subspace. We say that $V$ is a \emph{partition subspace} if there is some equivalence relation $\sim$ on $[n] = \{1,\ldots, n\}$ such that $V = \{ x\in \mathbb{R}^n \mid x_i = x_j \textrm{ whenever }i\sim j\}$.
\end{mydef}

\begin{mydef}\label{definition: equivalent reps}
Let $G$ be a group. 
\begin{enumerate}
    \item 
Let $M$ be a weakly free module of rank $n$ over a zero-sum-free semifield $K$ and let $\alpha, \beta: G\rightarrow \mathrm{Aut}(M)$ be linear $G$-actions. We say that $\alpha, \beta$ are equivalent if there exists $\phi \in (K^\times)^n \cap \mathrm{Aut}(M)$ such that 
\[
\alpha(g) = \phi \beta(g) \phi^{-1}, \quad \forall g \in G.
\]
\item 
Similarly let $M\subseteq K^n$ be a $K$-linear space over an idempotent semifield $K$ and let $\alpha, \beta: G\rightarrow \mathrm{Aut}(V)$ be linear $G$-actions.  We say that $\alpha, \beta$ are equivalent if there exists $\phi \in (K^\times)^n \cap \mathrm{Aut}(V)$ such that $\alpha(g) = \phi \beta(g) \phi^{-1}$ for all $g\in G$.
\end{enumerate}
\end{mydef}

We note that equivalence of actions is stronger than isomorphism - we require that there exists a diagonal isomorphism preserving the actions.

Now, we state our main theorem in this section. To this end, suppose $M$ and $\psi$ are as in one of the following cases:
\begin{enumerate}[label=(\alph*)]
    \item 
Let $M$ be a weakly free module of rank $n$ over $\T$ such that $M$ is finitely presented and $M$ can be embedded into a free $\T$-module of finite rank.  Let $\psi: \mathrm{Aut}(M)\rightarrow S_n$ be the map describing the action of automorphisms on weak basis lines. 
    \item 
Let $M$ be a weakly free module of rank $n$ over $\mathbb{R}_{\geq 0}$.  Let $\psi: \mathrm{Aut}(M)\rightarrow S_n$ be the map describing the action of automorphisms on weak basis lines.
    \item 
Let $M\subseteq \T^n$ be a $\mathbb{T}$-linear space such that $M$ is finitely generated and can be written as an intersection of only finitely many $\mathbb{T}$-hyperplanes.  Let $\psi: \mathrm{Aut}(M)\rightarrow S_n$ be the composition of the inclusion into $\text{GL}_n(\T)$ with the map $\text{GL}_n(\T)\rightarrow S_n$.

\end{enumerate}

In each of the above cases, we let $H$ be the image of $\psi$.

\begin{mythm}\label{theorem: automorphism group as a semidirect product}
With the notation as above, we have the following. 
\begin{enumerate}
    \item 
$\mathrm{Aut}(M) \cong H\ltimes V$ for some partition subspace $V\subseteq \mathbb{R}^n$. Moreover, the action of $H$ on $V$ is the permutation action induced by the action of $S_n$ on $\mathbb{R}^n$.\footnote{Note that part of the claim is that $V\subseteq \mathbb{R}^n$ is closed under the permutation action of $H$ on $\mathbb{R}^n$.}  
\item
If $G$ is a finite group, then composition with $\pi: \mathrm{Aut}(M)\rightarrow H$ (where $\pi$ is the corestriction of $\psi$) yields a one-to-one correspondence between equivalence classes (as in Definition \ref{definition: equivalent reps}) of linear $G$-actions on $M$ and homomorphisms $G\rightarrow H$.
\end{enumerate}
\end{mythm}
\begin{proof}
(1): We first prove the cases (a) and (c). Consider the following:
\[
V=(\mathbb{T}^\times)^n \cap (\mathrm{Aut}(M))=\mathbb{R}^n \cap (\mathrm{Aut}(M)).
\]
By Lemma \ref{lemma: Decomposition of quasi-free automorphisms}, we have a short exact sequence
\begin{equation}
1 \rightarrow V \rightarrow \mathrm{Aut}(M) \rightarrow H \rightarrow 1,
\end{equation}
and the action of $H$ on $V$ is the permutation action.\footnote{$H$ is a subgroup of $S_n$.} By Lemma \ref{lemma: Automorphism groups of weakly free modules are linear subgroups} for the case (a) or Lemma \ref{lemma: Automorphism groups of tropical spaces are linear subgroups} for the case (c), $V$ is a linear subgroup. Moreover, by Proposition \ref{proposition: tropical index}, we see that $V$ is isomorphic to a partition subspace of $\mathbb{R}^n$. 

For the second claim, by Proposition \ref{proposition: cohomology of representation}, we have
\[
H^i(H, V) = 0, \quad \forall i >0, 
\]
in particular, we have $H^2(H, V) = 0$. By Proposition \ref{proposition: H^2 and extensions}, we have that
\[
\mathrm{Aut}(M) \cong H\ltimes V.
\]

In the case (b), the argument is similar with $V=(\mathbb{R}_{\geq 0}^\times)^n \cap (\mathrm{Aut}(M))$, Lemma \ref{lemma: Automorphism groups of weakly free modules are linear subgroups}, and Proposition \ref{proposition: nonnegative index}. Note that we see $V$ as a real vector space by considering an isomorphism between $(\mathbb{R}_{\geq 0}^\times, \cdot)$ and $(\mathbb{R},+)$ via the logarithm map.

(2): There is a canonical map $\eta$ from the set of equivalence classes of linear $G$-actions on $M$ to homomorphisms $G\rightarrow H$, given by $\eta(\alpha) = \pi \circ \alpha$.  To see this is well-defined, suppose $\alpha$ and $\beta$ are equivalent, so there exists $\phi \in V = (K^\times)^n \cap \mathrm{Aut}(M)$ such that $\alpha(g) = \phi \beta(g) \phi^{-1}$.  Since $\phi\in V = \ker \pi$, this yields $\pi(\alpha(g)) = \pi(\beta(g))$ so $\eta(\alpha) = \eta(\beta)$ as desired.

To see that $\eta$ is injective, suppose that $\eta(\alpha)=\eta(\beta)$, i.e., $\pi\circ \alpha = \pi \circ \beta$. Let's fix $\alpha$. By Lemma \ref{lemma: cohomological classification of homomorphisms}, there is a bijection between the set of equivalence classes of homomorphisms $\beta: G \to \text{Aut}(M)$ with $\pi\circ \alpha=\pi\circ \beta$ and $H^1(G,V)$. However, by Proposition \ref{proposition: cohomology of representation}, $H^1(G, V) = 0$, and hence any such $\beta$ is equivalent to $\alpha$, showing that $\eta$ is injective.

Now, by Proposition \ref{proposition: H^2 and extensions}, $\pi: \mathrm{Aut}(M)\rightarrow H$ splits. Let $\iota: H\rightarrow \mathrm{Aut}(M)$ be a splitting of $\pi$.  Then for each homomorphism $\phi: G\rightarrow H$, we obtain an action $\iota\phi:G \rightarrow \mathrm{Aut}(M)$.  Since $\eta(\iota\phi) = \pi\iota\phi = \phi$, $\eta$ is surjective.
\end{proof}

 By combining Lemma \ref{lemma: Base change of quasi-free module to B version 2} with part (a) or (b), we obtain the following. 

\begin{cor}\label{cor: (d)}
Let $M$ be as in part (a) or (b).  Suppose in addition that $M$ is quasi-free.  Let $\psi: \mathrm{Aut}(M)\rightarrow\mathrm{Aut}(M\otimes_K \mathbb{B})$. Then, the same conclusions as in Theorem \ref{theorem: automorphism group as a semidirect product} hold. 
\end{cor}

Let $M$ be a polyhedral cone that contains no lines and let $W$ be the vector space it spans.  Let $n$ be the number of extreme rays of $M$. Call a permutation $\sigma$ of the extreme rays of $M$ \emph{realizable} if there is some invertible map $T: W\rightarrow W$ such that $T$ maps the $i$th ray to the $\sigma(i)$th ray for all $i$.

Let $G$ be a finite group.  Call two linear $G$-actions on $W$ which preserve $M$ \emph{equivalent} if there is an isomorphism $T: W \rightarrow W$ of representations such that any vector pointing along an extreme ray of $M$ is an eigenvector of $T$. 

We note that the semidirect product part of the following corollary already appears as the main result of \cite{horne1978automorphism}.

\begin{cor}\label{corollary: cone case}
With the same notation as above, let $H$ be the group of realizable permutations of the extreme rays of $M$.  Then the following hold.
\begin{enumerate}
    \item 
    There is a partition subspace $V\subseteq \mathbb{R}^n$ such that the group of invertible linear maps $T: W \rightarrow W$ which preserve $M$ (meaning $T(M) = M$) is isomorphic to $H \ltimes V$ where $H$ acts on $V$ via the permutation action.
    \item 
  There is a one-to-one correspondence between equivalence classes of linear $G$-actions on $W$ which preserve $M$ and homomorphisms $G\rightarrow H$.
\end{enumerate}
\end{cor}
\begin{proof}
We first note that $M$ is quasi-free; a set consisting of one vector from each extreme ray forms a quasi-basis. In particular, by Lemma \ref{lemma: quasi-basis permute}, $M$ is weakly free.
Moreover, we can embed $M$ into $\mathbb{R}_{\geq 0}^m$ by taking the dot product with each extreme ray of the dual cone. The definition of dual cone makes this non-negative, and the dual cone spans the dual space since $M$ contains no lines.

Now, since $M$ is quasi-free, we can apply Theorem \ref{theorem: automorphism group as a semidirect product} to $M$. 

To be precise, for (1), in Theorem \ref{theorem: automorphism group as a semidirect product} (in the case (b)), the image of $\psi:\text{Aut}(M) \to S_n$ is precisely the group of realizable permutations of the extreme rays of $M$. Also, the group of invertible linear maps $T:W \to W$ which preserves $M$ is precisely $\text{Aut}(M)$.

For (2), once we fix a quasi-basis consisting of extreme rays, $T$ is equivalent to $T'$ if and only if they are conjugate by a diagonal matrix. Now, under the identification between $\text{Aut}(M)$ and the group of invertible linear maps $T:W \to W$ which preserves $M$, this is equivalent to Definition \ref{definition: equivalent reps}, and hence Theorem \ref{theorem: automorphism group as a semidirect product} applies to obtain the desired result. 
\end{proof}

\begin{cor}
With the same notation as above, the action on $M$ is determined by the action on $M\otimes_{\mathbb{R}_{\geq 0}}\mathbb{B}$.
\end{cor}
\begin{proof}
This directly follows from Lemma \ref{lemma: Base change of quasi-free module to B version 2}  and Corollary \ref{corollary: cone case}.
\end{proof}

\section{Valuated matroidal representations of groups}\label{section: valuated matroids}

In this section we will use the (equivalent) definition of tropical linear spaces given in \cite{frenk2013tropical}. We note that \cite{frenk2013tropical} uses the minimum convention whereas we use the maximum convention, but both conventions produce isomorphic semifields. We will assume that all matroids are simple unless otherwise stated. 

Let $(M,w)$ be a valuated matroid on $[n]$ with a valuation $w:\binom{[n]}{d} \to \mathbb{T}$, where $d$ is the rank of the underlying matroid $\underline{M}$ of $M$. One can think of $\underline{M}$ as a valuated matroid with the valuation $\pi\circ w$, where $\pi:\mathbb{T} \to \mathbb{B}$ sends any nonzero element of $\mathbb{T}$ to $1_\mathbb{B}$ and $0_\mathbb{T}$ to $0_\mathbb{B}$. We denote by $\mathcal{B}$ the set of bases of $M$. In this section, we will often write $M$ instead of $(M,w)$ whenever this does not create an ambiguity.

Suppose that $\text{rk}(M)=d$. For each $I \subseteq [n]$, with $|I|=d+1$, let $V_I(w)$ be the tropical hyperplane defined by the following linear equation:
\begin{equation}\label{eq: tropical linear space hyperplane}
\sum_{j \in I} w(I - j)x_j,
\end{equation}
where $w(I-j)=0_{\TT}$ if $I -j \not \in \mathcal{B}$. Now, we define
\[
V_M:=\bigcap_{I}V_I(w).
\]
$V_M$ is the tropical linear space associated to a valued matroid $(M,w)$. 

By a \textit{weak automorphism} of a valuated matroid on $[n]$, we mean an element $\sigma \in S_n$ such that there exists a map $\tau:[n] \to \mathbb{T}$ such that for each $B \in \mathcal{B}$, one has
\[
w(\sigma(B))=(\prod_{i \in B}\tau(i))w(B).
\]

\begin{rmk}
In \cite{dress1992valuated} a notion of projective equivalence between two valuated matroids with the same underlying set $E$ was introduced as follows. 

Let $\underline{M}$ be a matroid on $E=[n]$ of rank $d$. Two valuations $w, w'$ of $\underline{M}$ are said to be \emph{projectively equivalent} if there exist some $\alpha \in \mathbb{R}=\mathbb{T}^\times$ and a map $\tau:E \to \mathbb{T}$ such that for each basis $B=\{b_1,b_2,\dots,b_d\}$ of $\underline{M}$, one has
\[
w'(B)=\alpha\cdot \prod_{i=1}^d\tau(b_i) w(B).
\]
This clearly defines an equivalence relation on the set of valuated matroid structures on $\underline{M}$. In other words, we can equivalently define a weak automorphism of a valuated matroid $(M,w)$ to be an element $\sigma \in S_n$ such that $w \circ \sigma$ and $w$ are projectively equivalent. 
\end{rmk}

Let $(M,w)$ be a valuated matroid on $[n]$ and $V_M$ be the corresponding tropical linear space in $\mathbb{T}^n$. 

In the following by the underlying module of $V_M$, we mean simply $V_M$ as a module in contrast to $V_M$ as a tropical linear space which we view as a tuple of $(V_M,\mathbb{T}^n)$, where $V_M \subseteq \mathbb{T}^n$. 

In the following Proposition~\ref{prop: V_M} we recall results about tropical linear spaces, combining \cite[Lemma 4.1.4]{frenk2013tropical} and \cite[Corollary 4.1.10]{frenk2013tropical}. Recall that a hyperplane of a matroid is a flat whose rank is 1 less than the rank of the matroid. 

\begin{pro}\label{prop: V_M}
\label{pro: Generators of tropical linear space} Let $M$ be a valuated matroid on $[n]$ and $V_M\subseteq\T^n$ be the corresponding tropical linear space as defined above.
For each independent set $I$ of corank $1$, let $v_I\in \T^n$ be given by 
\begin{equation}\label{eq: valuation}
(v_I)_i = 
\begin{cases}
w(I+i), \quad \textrm{if $i \not \in I$},\\
0_\mathbb{T}, \quad \quad \quad  \ \  \textrm{if $i \in I$.}
\end{cases}
\end{equation}
\begin{enumerate}
    \item 
$V_M$ is generated as a $\T$-module by vectors of the form $v_I$ where $I$ ranges over corank $1$ independent sets.
    \item 
The support of $v_I$ is the complement of the hyperplane spanned by $I$.
    \item 
Let $\mathcal{J}$ be a collection of corank $1$ independent sets.  For any corank $1$ independent set $I$, we have $v_I \in \mathrm{span}(\{v_J \mid J\in \mathcal{J}\})$ if and only if $I$ spans the same hyperplane as some $J\in \mathcal{J}$.  In particular $v_I$ is a scalar multiple of $v_J$ if and only if $I$ and $J$ span the same hyperplane.
\end{enumerate}
\end{pro}

\begin{lem}\label{lemma: matroids are determined by hyperplanes}
Let $M$ be a matroid of rank $d$ on a finite set $E$.  A subset $B\subseteq E$ is a basis if and only if it has $d$ elements and is not contained in any hyperplane of $M$.  Consequently an invertible map $\phi: E \rightarrow E$ is a matroid automorphism if and only if both $\phi$ and $\phi^{-1}$ map hyperplanes to hyperplanes.
\end{lem}
\begin{proof}
For the first claim, for any $S \subseteq E$ we have that $\mathrm{rank}(S) = \mathrm{rank}(\bar{S})$, where $\bar{S}$ denotes the smallest flat containing $S$.  Let $B\subseteq E$ contain $d$ elements and suppose $B$ is not contained in any hyperplane.  Then $\bar{B}$ is not contained in any hyperplane, and hence $\mathrm{rank}(B) = \mathrm{rank}(\bar{B}) > d - 1$.  On the other hand, $\mathrm{rank}(B) \leq d$ so $\mathrm{rank}(B) = d = |B|$ which implies $B$ is independent and hence a basis.

Conversely, if $B$ is a basis, then any flat containing $B$ consists of all of $E$, so $B$ is not contained in a hyperplane.  And clearly $B$ has $d$ elements.

The second claim directly follows from the first claim. 
\end{proof}

\begin{lem}\label{lemma: weak correspondence between corank 1 independent sets}
Let $M$ be a valuated matroid on $[n]$ and $V_M$ be the associated tropical linear space in $\T^n$.  Let $\sigma\in S_n$ be in the image of $\mathrm{Aut}(V_M) \subseteq \text{GL}_n(\T)$ under the map $\text{GL}_n(\T) \rightarrow S_n$.  Then for every independent set $I$ with corank $1$, there exist $a, c_1,\ldots, c_n \in \T^\times$ and a corank $1$ independent set $J$ such that for all $k$ we have
\begin{equation}\label{eq: weak correspondence between corank 1 independent sets}
w(J \cup \{\sigma(k)\}) = a c_k w(I \cup \{k\})
\end{equation}
Moreover, we can choose $c_1,\ldots, c_n \in \T^\times$ to not depend on $I$.
\end{lem}
\begin{proof}

Consider two corank $1$ independent sets to be equivalent if they span the same hyperplane, and let $S$ contain one representative for each equivalence class.  For each corank $1$ independent set $I$, we let $v_I$ be as in Proposition~\ref{prop: V_M}. Then, it follows from Proposition \ref{pro: Generators of tropical linear space} that each $v_I$ is a $\T$-multiple of $v_J$ for some $J\in S$ and $V_M$ is generated by elements of the form $v_I$.  So $\{v_J \mid J\in S\}$ is a generating set, and by another application of Proposition \ref{pro: Generators of tropical linear space}, we see that no element of $\{v_J \mid J\in S\}$ is a linear combination of other elements. Thus we have found a minimal generating set.  By Proposition \ref{proposition: Wagneurs result}, $V_M\subseteq \mathbb{T}^n$ is weakly free as a $\mathbb{T}$-module.  Let $f\in \mathrm{Aut}(V_M)$ map to $\sigma\in S_n$.  Then $f$ induces a module automorphism of the underlying module of $V_M$ and so preserves the minimal generating set up to rescaling and permutation.  So for $I\in S$ there is some $J\in S$ such that $f(v_I) = r_{IJ}v_J$ for some $r_{IJ} \in \mathbb{T}^\times$. 

We first prove the result in the case $I\in S$, and we let $J$ be as above.  Because $f\in \text{GL}_n(\T)$ maps to $\sigma\in S_n$, there exist $c_1,\ldots, c_n$ such that $f(e_k) = c_k e_{\sigma(k)}$.  This yields
\begin{equation}\label{eq: weak correspondence between corank 1 independent sets2}
r_{IJ}v_J = f(v_I) = f(\sum_k w(I\cup \{k\}) e_k) = \sum_k c_k w(I \cup \{k\}) e_{\sigma(k)},
\end{equation}
where $w(I\cup\{k\})=0_\mathbb{T}$ if $k \in I$. Extracting the $\sigma(k)$ coordinate yields
\begin{equation}
r_{IJ}w(J \cup \sigma(k)) = c_k w(I \cup \{k\}).
\end{equation}
For a general corank $1$ independent set $I$, there exists $I'\in S$ which spans the same hyperplane.  By Proposition \ref{pro: Generators of tropical linear space}, there exists $t\in \mathbb{T}^\times$ such that $v_I = tv_{I'}$. Then, by what we have proven above, there exist $c_1,\ldots c_n$ and $J$ such that \eqref{eq: weak correspondence between corank 1 independent sets2} holds with $I'$ in place of $I$ so
\begin{equation}
r_{IJ}w(J \cup \sigma(k)) = c_k w(I' \cup \{k\}) = tc_k w(I \cup \{k\}).
\end{equation}
\end{proof}

\begin{lem}\label{lemma: induction}
Let $M$ be a valuated matroid on $[n]$ and $V_M$ be the associated tropical linear space in $\mathbb{T}^n$.  Let $\sigma\in S_n$ be in the image of $\mathrm{Aut}(V_M)$ under the map $\text{GL}_n(\T) \rightarrow S_n$.  Then $\sigma: [n]\rightarrow [n]$ is an automorphism of the underlying matroid $\underline{M}$.  Furthermore, for every corank $1$ independent set $I$, there exists $\lambda_I\in \T^\times$ such that for all $k$,
\begin{equation}
w(\sigma(I) \cup \{\sigma(k)\}) = \lambda_Ic_k w(I \cup \{ k \}),
\end{equation}
where the $c_k$'s are the same as in the statement of Lemma \ref{lemma: weak correspondence between corank 1 independent sets}.
\end{lem}
\begin{proof}To show $\sigma$ is a matroid automorphism it suffices by Lemma \ref{lemma: matroids are determined by hyperplanes} to show $\sigma$ and $\sigma^{-1}$ map hyperplanes to hyperplanes. 

Given a hyperplane $H$, we choose a maximal independent subset $I$, which clearly has corank $1$.  By Lemma \ref{lemma: weak correspondence between corank 1 independent sets}, there exists a corank $1$ independent set $J$ and a unit $a\in \T^\times$ such that
\begin{equation}
w(J \cup \{\sigma(k)\}) = ac_k w(I \cup \{k\}).
\end{equation}
The right side is $0_\mathbb{T}$ if and only if $k\in H$.  Hence $w(J \cup \{\sigma(k)\})=0_\mathbb{T}$ if and only if $k\in H$.  Consequently $w(J\cup \{k\}) = 0_\mathbb{T}$ if and only if $k\in \sigma(H)$.  On the other hand the set of $k$ for which $w(J\cup \{k\}) = 0_\mathbb{T}$ is the hyperplane spanned by $J$.  Consequently $\sigma(H)$ is a hyperplane as desired.  Since $\sigma^{-1}$ satisfies the hypotheses of this lemma, we may apply the above to $\sigma^{-1}$ to see that $\sigma^{-1}(H)$ is also a hyperplane.  This establishes that $\sigma$ is a matroid automorphism.

Now, because $\sigma$ is a matroid automorphism, $\sigma(I)$ is a corank $1$ independent set.  Similarly for each $k$, $\sigma(I) \cup \{\sigma(k)\}$ fails to be a basis if and only if $I\cup\{k\}$ fails to be a basis.  This means that $\sigma(k)$ belongs to the hyperplane spanned by $\sigma(I)$ if and only if $k\in H$.  But, we proved that this holds if and only if $w(J \cup \{\sigma(k)\})=0_\mathbb{T}$, i.e. if and only if $\sigma(k)$ is in the hyperplane spanned by $J$.  So $\sigma(I)$ and $J$ span the same hyperplane. Consequently, there exists some $b\in \T^\times$ such that for all $k$,
\begin{equation}w(\sigma(I) \cup \{\sigma(k)\}) = bw(J \cup \{\sigma(k)\}) = abc_k w(I \cup \{k\}).
\end{equation}
\end{proof}

\begin{pro}\label{proposition: val weak}
Let $M$ be a valuated matroid on $[n]$ and $V_M$ be the associated tropical linear space in $\T^n$.  Let $\sigma\in S_n$ be in the image of $\mathrm{Aut}(V_M)$ under the map $\text{GL}_n(\T) \rightarrow S_n$.  There exist elements $\lambda, c_1,\ldots, c_n\in\mathbb{T}^\times$ such that for any subset $B$ of size $\mathrm{rank}(M)$,
\begin{equation}
w(\sigma(B)) = (\lambda \prod_{k\in B}c_k) w(B)
\end{equation}
\end{pro}
\begin{proof}
Let $c_1,\ldots, c_n$ be as in Lemma \ref{lemma: weak correspondence between corank 1 independent sets}.  We will first prove the result when $B$ is a basis.  Pick a basis $B_0$, and let
\[
\lambda := \frac{w(\sigma(B_0))}{w(B_0)\prod\limits_{k\in B_0}c_k}.
\]
Then clearly the result holds in the case that $B=B_0$.  We proceed by induction on $|B\triangle B_0|$.

We may choose $x\in B_0 \backslash B$ and $y\in B \backslash B_0$ such that $B':=B \cup \{ x\} \backslash \{y\}$ is a basis.  Let $I = B \backslash \{y\}$, so $B = I \cup \{y\}$ and $B' = I\cup \{x\}$.  Then, by Lemma \ref{lemma: induction}, there exists $\lambda_I$ such that
\begin{equation}
w(\sigma(B')) = w(\sigma(I) \cup \{\sigma(x)\}) = \lambda_I c_x w(I \cup \{x\}) = \lambda_I c_x w(B')
\end{equation}
and
\begin{equation}
w(\sigma(B)) = w(\sigma(I) \cup \{\sigma(y)\}) = \lambda_I c_y w(I \cup \{y\}) = \lambda_I c_y w(B).
\end{equation}
Combining the last two equations yields
\begin{equation}
\frac{w(\sigma(B))}{w(B)} = \lambda_I c_y = \frac{c_y}{c_x}\frac{w(\sigma(B'))}{w(B')}.
\end{equation}
On the other hand, by the inductive hypothesis,
\begin{equation}
w(\sigma(B')) = (\lambda \prod_{k\in B'}c_k) w(B') = (\lambda \prod_{k\in B}c_k) \frac{c_x}{c_y} w(B').
\end{equation}
This yields
\begin{equation}
\frac{w(\sigma(B))}{w(B)} = \lambda \prod_{k\in B}c_k
\end{equation}
which is the desired result in the case that $B$ is a basis.  

Finally, suppose $B$ is not a basis.  Since Lemma \ref{lemma: induction} shows $\sigma$ is a matroid automorphism of $\underline{M}$, $\sigma(B)$ is not a basis.  Thus $w(B)$ and $w(\sigma(B))$ are both zero, which implies the result in this case.
\end{proof}

\begin{mythm}\label{theorem: tropical linear space automorphisms induce weak automorphisms}
Let $M$ be a valuated matroid on $[n]$ with the associated tropical linear space $V_M$ in $\T^n$.  Let $H\subseteq S_n$ be the image of $\mathrm{Aut}(V_M)$ under the map $\text{GL}_n(\T) \rightarrow S_n$.  Then $H$ is the weak automorphism group of $M$.
\end{mythm}
\begin{proof}
First let $\sigma \in H$.  By Proposition \ref{proposition: val weak}, there exists $\lambda, c_1,\ldots, c_n$ such that for every subset $B$ of size $\mathrm{rank}(M)$,
\begin{equation}
w(\sigma(B)) = (\lambda \prod_{k\in B}c_k) w(B).
\end{equation}
If we take $c_k' = \lambda^{1/\mathrm{rank}(M)}c_k$, then
\begin{equation}
w(\sigma(B)) = (\prod_{k\in B}c_k') w(B),
\end{equation}
and hence $\sigma$ is a weak automorphism.

Conversely, let $\sigma$ be a weak automorphism.  Then there exist $c_1,\ldots,c_n$ such that for every subset $B$ of size $\mathrm{rank}(M)$,
\begin{equation}\label{eq: weak}
w(\sigma(B)) = (\prod_{k\in B}c_k) w(B).
\end{equation}
Now consider $f \in \text{GL}_n(\mathbb{T})$ given by $f(e_k)=c_ke_{\sigma(k)}$. We show that $f \in \text{Aut}(V_M)$. This will imply that $\sigma \in H$. So, we prove that for $x \in V_M$, we have $f(x) \in V_M$ and conversely, for any $z \in V_M$, we have $f^{-1}(z) \in V_M$.

Let $S\subseteq M$ have $\mathrm{rank}(M) + 1$ elements.  Then the same is true of $\sigma(S)$.  By the definition of $V_M$, for any $x\in V_M$, the bend relations of the following 
\begin{equation}\label{eq: now}
\sum_{j\in \sigma(S)} w(\sigma(S) -j) x_j 
\end{equation}
are satisfied by any such $S$. Now, we can rewrite \eqref{eq: now} as follows:
\begin{equation}
\sum_{i\in S} w(\sigma(S) -\sigma(i)) x_{\sigma(i)}.
\end{equation}
From \eqref{eq: weak}, we can rewrite the terms in the above sum as
\begin{equation}\label{eq: pre-division}
\sum_{i\in S} (\prod_{k\in S-i}c_k) w(S - i) x_{\sigma(i)}.
\end{equation}
Of course the bend relations of (\ref{eq: pre-division}) will be unchanged if we divide every term by $\prod_{k\in S}c_k$, so the bend relations of
\begin{equation}\label{eq: after-division}
\sum_{i\in S} w(S - i) \frac{x_{\sigma(i)}}{c_i}
\end{equation}
are the same as the bend relations of (\ref{eq: pre-division}). Define $y\in \T^n$ by $y_i = x_{\sigma(i)} / c_i$, i.e., $x=f(y)$. Since the bend relations of
\begin{equation}
\sum_{i\in S} w(S- i) y_i
\end{equation} 
are satisfied by any $S$ of size $\mathrm{rank}(M) + 1$, $y=f^{-1}(x)\in V_M$. Now, one can reverse the above to see that $f(V_M) \subseteq V_M$. 
\end{proof}

\begin{mydef}
Let $M$ be a valuated matroid and $G$ be a group. 
\begin{enumerate}
    \item 
A weak $G$-action on $M$ is a homomorphism $G\rightarrow \mathrm{Aut}_w(M)$, where $\mathrm{Aut}_w(M)$ denotes the weak automorphism group of $M$. 
\item 
 Two weak $G$-actions $\alpha, \beta: G\rightarrow\mathrm{Aut}_w(M)$ are weakly isomorphic if there is a weak automorphism $\phi \in \mathrm{Aut}_w(M)$ such that $\phi\alpha(g) = \beta(g)\phi$ for all $g\in G$.
\end{enumerate}
\end{mydef}

\begin{rmk}\label{remark: equivalent vs isomoprhic}
Let $V_M$ be a tropical linear space in $\mathbb{T}^n$ associated to a valuated matroid $M$.
\begin{enumerate}
    \item 
Recall that tropical subrepresentations $\alpha,\beta:G \to \text{Aut}(V_M)$ are \emph{equivalent}, denoted by $\alpha \sim \beta$, if there is some $D \in \text{Aut}(V_M) \cap (\mathbb{T}^\times)^n$ such that
\begin{equation}\label{eq: equivalent}
\alpha(g)D =D\beta(g), \quad \forall g \in G. 
\end{equation}
On the other hand, $\alpha$ and $\beta$ are \emph{isomorphic} if there is some $F \in \text{Aut}(V_M)$ such that 
\begin{equation}\label{eq: isomorphic}
\alpha(g)F =F\beta(g), \quad \forall g \in G.
\end{equation}
\item 
Since Theorem \ref{theorem: automorphism group as a semidirect product} states $\mathrm{Aut}(V_M)$ is the semidirect product of $\mathrm{Aut}(V_M)\cap (\T^\times)^n$ and $H$, we view $H$ as a subgroup of 
$\mathrm{Aut}(V_M)$, so each element of $\mathrm{Aut}(V_M)$ can be uniquely written as $DP$ with $D\in \mathrm{Aut}(V_M)\cap (\T^\times)^n$ and $P\in H$.  And $DP = PD'$ for some $D'\in \mathrm{Aut}(V_M)\cap (\T^\times)^n$. 
\end{enumerate}
\end{rmk}

Let $H$ be the image of $\mathrm{Aut}(V_M)$ inside $S_n$. By Theorem \ref{theorem: tropical linear space automorphisms induce weak automorphisms}, we have $H = \mathrm{Aut}_w(M)$.  Theorem \ref{theorem: automorphism group as a semidirect product} implies that there is a one-to-one correspondence between equivalence classes of homomorphisms $G\rightarrow \mathrm{Aut}(V_M)$ under the equivalence relation in \eqref{eq: equivalent} and homomorphisms $G\rightarrow \mathrm{Aut}_w(M)$.  However, observe that the notion of the equivalence \eqref{eq: equivalent} is weaker than requiring the homomorphisms to be isomorphic as in \eqref{eq: isomorphic}. Nonetheless, we have the following.

\begin{cor}\label{corollary: main cor}
Let $M$ be a valuated matroid, $V_M$ be the associated tropical linear space in $\T^n$, and $G$ be a finite group.  Then isomorphism classes of tropical subrepresentations whose underlying tropical linear space is isomorphic to $V_M$ are in one-to-one correspondence with weak isomorphism classes of weak $G$-actions on $M$.
\end{cor}
\begin{proof}
It is clear that isomorphism classes of tropical subrepresentations whose underlying tropical space is isomorphic to $V_M$ are in one-to-one correspondence with isomorphism classes of homomorphisms $G\rightarrow \mathrm{Aut}(V_M)$. Let $H$ be the image of $\mathrm{Aut}(V_M)$ inside $S_n$. As mentioned in Remark \ref{remark: equivalent vs isomoprhic}, we view $H$ as a subgroup of $\text{Aut}(V_M)$ via the semidirect product $\text{Aut}(V_M)=H\ltimes V$, where $V=\mathrm{Aut}(V_M)\cap (\T^\times)^n$.

In the following, we let $A$ be the set of homomorphisms $G \to \text{Aut}(V_M)$, and $X$ (resp.~$Y$) be the set of equivalence classes (resp.~isomorphism classes) of $A$.

Note that $H$ acts on $A$ by conjugation, where $H$ is viwed as a subgroup of $\text{Aut}(V_M)$. We first show that $H$-action on $A$ descends to $X$ and that $Y$ is the set of orbits under this action. To see this, first observe that for $\alpha, \beta \in A$, if $\alpha \sim \beta$, then there exists $D \in \mathrm{Aut}(V_M)\cap (\T^\times)^n$ such that 
\[
\alpha(g)D = D\beta(g), \quad \forall g\in G. 
\]
Let $P \in H$. Then, there exists $D' \in \mathrm{Aut}(V_M)\cap (\T^\times)^n$ such that $PD=D'P$.\footnote{See Remark \ref{remark: equivalent vs isomoprhic}.} It follows that
\[
(P\alpha(g)P^{-1})D' = (P\alpha(g))(DP^{-1})=P(D\beta(g))P^{-1}=D'(P\beta(g)P^{-1}),
\]
showing that $P\alpha P^{-1} \sim P\beta P^{-1}$. So $H$-action on $A$ by conjugation descends to $X$.  

Moreover since each element of $\mathrm{Aut}(V_M)$ has the form $DP$ for $P \in H$ and $D \in \mathrm{Aut}(V_M)\cap (\T^\times)^n$, $\alpha$ and $\beta$ are isomorphic if and only if there exist such $D$ and $P$ such that 
\[
\alpha(g)DP = DP\beta(g),
\]
or equivalently if $\alpha$ is equivalent to $P\beta P^{-1}$.  This shows isomorphism classes correspond to orbits of equivalence classes.

Let $Z$ be the set of homomorphisms $G \to \text{Aut}_w(M)$. Observe that $H=\mathrm{Aut}_w(M)$ acts on $Z$ by conjugation and acts on $X$ also by conjugation, but with $H$ viewed as a subgroup of $\text{Aut}(V_M)$. It follows from Theorem \ref{theorem: automorphism group as a semidirect product} that the following is a bijection:
\[
\eta: X \to Z, \quad \alpha \mapsto \pi\circ \alpha,
\]
where $\pi$ is the restriction of the projection map $\pi:\text{GL}_n(\mathbb{T}) \to S_n$ to $\text{Aut}(V_M)$. We claim that the bijection $\eta$ is $H$-equivariant when we view $X$ and $Z$ as $H$-sets. To make the notation less confusing, we let $\iota:H \to \text{Aut}(V_M)$ be the section of $\pi:\text{Aut}(V_M) \to H$. Then the $H$-action on $X$ is as follows:
\[
h\cdot [\alpha] = [\iota(h)\alpha \iota(h^{-1})],
\]
where $[\alpha]$ is the equivalence class of $\alpha \in A$ in $X$. Now, for any $h \in H$ and $[\alpha] \in X$, we have
\[
\eta(h\cdot [\alpha])=\eta([\iota(h)\alpha \iota(h^{-1})])=\pi  (\iota(h)\alpha \iota(h^{-1}))
\]
\[
=(\pi (\iota(h)) (\pi(\alpha)\pi  (\iota(h^{-1}))=h\pi(\alpha)h^{-1}=h\cdot \eta([\alpha]),
\]
showing that $\eta$ is $H$-equivariant. 

Since the set of homomorphisms $G\rightarrow \mathrm{Aut}_w(M)$ and the set of equivalence classes of homomorphisms $G\rightarrow \mathrm{Aut}(V_M)$ agree as $H$-sets, they have the same set of orbits.  And we saw that the orbits of the set of equivalence classes of homomorphisms $G\rightarrow \mathrm{Aut}(V_M)$ are precisely the isomorphism classes of homomorphisms $G\rightarrow \mathrm{Aut}(V_M)$.  On the other hand, by definition the orbits of homomorphisms $G\rightarrow \mathrm{Aut}_w(M)$ are precisely the weak isomorphism classes of $G\rightarrow \mathrm{Aut}_w(M)$.
\end{proof}

\begin{myeg}
The above result is false for infinite groups.  For example, let $G = \T^\times$ and let $M$ be a valuated matroid. Then there exists a nontrivial $G$-action on $V_M$ (e.g. by scalar multiplication).  But $G$ cannot act nontrivially on $M$ because $G$ is divisible and $\text{Aut}_w(M)$ is finite.
\end{myeg}

In addition to tropical linear spaces, we may also use a quotient construction to associate a module to a valuated matroid. Let $M$ be a valuated matroid with rank $d$. Recall the tropical linear space $V_M$ is the subspace of $\T^n$ cut out by the bend relations associated to vectors of the form
\begin{equation}
\sum_{j\in S} w(S - j) x_j
\end{equation}
for each $S\subseteq M$ of size $d + 1$.  The quotient module $Q_M$ is instead the quotient of $\T^n$ by such bend relations; explicitly we impose for each $S$ of size $d + 1$ and each $k\in S$ that
\begin{equation}
\sum_{j\in S} w(S - j) e_j \sim \sum_{j\in S - k} w(S - j) e_j
\end{equation}

With the same notation as above, it was proved in \cite{giansiracusa2018grassmann} that 
\begin{equation}\label{QM}
Q_M^\vee:=\Hom(Q_M,\mathbb{T}) = V_M.
\end{equation}

Note that in general $Q_M$ is not reflexive (see \cite{giansiracusa2018grassmann}), however it is quasi-free as the following lemma shows. 

\begin{lem}\label{lemma: QM quasi-free lemma}
With the same notation as above, $Q_M$ is quasi-free.  
\end{lem}
\begin{proof}
By Example \ref{ex:monoid-val-matroid}, it suffices to show that for any set $S$ of size $d + 1$, either $w(S - j) = 0_{\T}$ for all $j\in S$ or there are at least $3$ choices of $j\in S$ such that $w(S - j) \neq 0$.  For this, it suffices to show $\{ j\in S \mid w(S - j) \neq 0_{\T}\}$ is either empty or contains a circuit.

Suppose for contradiction that $I :=\{ j\in S \mid w(S - j) \neq 0_{\T}\}$ is a nonempty independent set.  Fix some $j_0 \in I$ and let $B = S - j_0$.  Then $B\subseteq S$ is a basis by definition of $I$.  Let $B'\subseteq S$ be a maximal independent subset of $S$ containing $I$.  Since all maximal independent subsets of $S$ have the same size, $B'$ also has $d$ elements and hence is a basis.  So there is exactly one element of $S$ (which we denote $k$) such that $k\not \in B'$.  Then $S - k$ is a basis so $w(S - k) \neq 0$, contradicting $k\not\in I$.
\end{proof}

From Lemma \ref{lemma: QM quasi-free lemma}, we may view $\text{Aut}(Q_M)$ as a subgroup of $\text{GL}_n(\mathbb{T})$. The following proposition shows that we can use $Q_M$ to compute the weak automorphism group $\text{Aut}_w(M)$.

\begin{pro}\label{proposition: Q_M case}
The image of $\mathrm{Aut}(Q_M)$ under the map $\pi:GL_n(\T)\rightarrow S_n$ is the weak automorphism group $\mathrm{Aut}_w(M)$.
\end{pro}
\begin{proof}

Let $\sigma$ belong to the image of $\mathrm{Aut}(Q_M)$ under $\pi:\text{GL}_n(\T)\rightarrow S_n$.  We first show that $\sigma$ is a weak automorphism of $M$.  Let $f\in \mathrm{Aut}(Q_M)$ be some element that maps to $\sigma$.  Explicitly this means there exists some $A\in \text{GL}_n(\T)$ which maps to $\sigma$ satisfying the following:
\begin{equation}\label{eq: commuting diagram}
\begin{tikzcd}
\mathbb{T}^n  \arrow[r,"\phi"] \arrow[d," A",swap] & Q_M \arrow[d,"f"] \\
\mathbb{T}^n  \arrow[r,"\phi"] & Q_M
\end{tikzcd}
\end{equation}
where $\phi: \T^n \rightarrow Q_M$ is the quotient map. 

As noted in (\ref{QM}), the dual module $\mathrm{Hom}(Q_M, \T)$ of $Q_M$ is $V_M$, and one can also show that dualizing the quotient map $\phi:\T^n \rightarrow Q_M$ gives the canonical embedding $\iota:V_M \rightarrow \T^n$.\footnote{See \cite[Lemma 4.7]{JMTmatroidsPart1}.} With this, we dualize \eqref{eq: commuting diagram} and obtain the following diagram:
\begin{equation}\label{eq: dual commuting diagram}
\begin{tikzcd}
\mathbb{T}^n  & V_M \arrow[l,swap, "\iota"] \\
\mathbb{T}^n \arrow[u," A^\vee"]  & V_M \arrow[l,swap, "\iota"] \arrow[u,swap,"f^\vee"] 
\end{tikzcd}
\end{equation}

In particular, $f^\vee$ is the restriction to $V_M$ of $A^\vee$.  So $A^\vee \in \text{GL}_n(\T)$ actually belongs to $\mathrm{Aut}(V_M)$.  

Observe that $A^\vee$ is simply the transpose of $A$ and so maps to $\sigma^{-1}\in S_n$. Then, since $A^\vee \in \mathrm{Aut}(V_M)$ maps to $\sigma^{-1} \in S_n$, we obtain from Theorem \ref{theorem: tropical linear space automorphisms induce weak automorphisms} that $\sigma^{-1}$ is a weak automorphism.  Hence $\sigma$ is a weak automorphism.

Conversely, let $\sigma$ be a weak automorphism of $M$.  Then there exists $\tau: [n] \rightarrow \T$ such that for all $B$ of size $d$, one has
\begin{equation}
w(\sigma(B)) = (\prod_{i\in B} \tau(i)) w(B).
\end{equation}
Let $\tau(i)=\tau_i$. Since $\underline{M}$ is simple, we have $\tau_i \neq 0_\mathbb{T}$ for all $i \in [n]$. Define a linear map 
\[
f: \T^n \rightarrow \T^n, \quad e_i \mapsto \tau_i^{-1} e_{\sigma(i)}.
\]
Then for any $S$ of size $d + 1$,
\begin{equation}\label{eq: 11}
\sum_{j\in S} w(\sigma(S) - \sigma(j)) e_{\sigma(j)} = \sum_{j\in S} w(\sigma(S) - \sigma(j)) \tau_j f(e_j) = \sum_{j\in S} \tau_j (\prod_{i\in S - j}) \tau_i w(S - j) f(e_j).
\end{equation}
The same computation with the $k^{\text{th}}$ term deleted yields
\begin{equation}\label{eq: 22}
\sum_{j\in S - k} w(\sigma(S) - \sigma(j)) e_{\sigma(j)} = \sum_{j\in S - k} \tau_j (\prod_{i\in S - j}) \tau_i w(S - j) f(e_j).
\end{equation}
The left sides of the equations \eqref{eq: 11} and \eqref{eq: 22} may be rewritten as sums over $\sigma(S)$, for instance, as follows 
\[
\sum_{i\in \sigma(S)} w(\sigma(S) - i) e_i,
\]
and both of these sums are equivalent in $Q_M$ via the bend relations associated to $\sigma(S)$.

Hence in $Q_M$, the bend relations of (either side of the equation)
\begin{equation}
     (\prod_{i\in S} \tau_i) \sum_{j\in S} w(S - j) f(e_j) = \sum_{j\in S} w(\sigma(S) - \sigma(j)) e_{\sigma(j)}
\end{equation}
are satisfied for any $S \subseteq [n]$ of size $d+1$. Dividing out the constant term, we obtain
\begin{equation}
    \sum_{j\in S} w(S - j) f(e_j) \sim \sum_{j\in S - k} w(S - j) f(e_j),
\end{equation}
where $\sim$ is the congruence defining $Q_M$.  Thus for each of the pairs $(a, b)$ generating the congruence $\sim$, we have $f(a) \sim f(b)$.  Since we obtain a congruence $\equiv$ by defining $a \equiv b$ if $f(a) \sim f(b)$, we see that $\sim \subseteq \equiv$.  Hence $a \sim b$ implies $f(a) \sim f(b)$.  So $f$ descends to a well-defined map $Q_M \rightarrow Q_M$.  This map is of course an isomorphism, because we may apply the same argument to $f^{-1}$ (with $\sigma$ replaced by $\sigma^{-1}$ and $\tau$ replaced appropriately). This yields an element $f\in \mathrm{Aut}(Q_M)$ whose image is $\sigma\in S_n$.
\end{proof}

\begin{myeg}
The requirement that $\underline{M}$ be simple cannot be dropped.  For example, let $M$ be the uniform valuated matroid of rank $1$, i.e. $w(B) = 1$ for every singleton $B$.  Clearly $\mathrm{Aut}_w(M) = S_n$.  The bend relations defining $Q_M$ have the form $e_i \sim e_j$ so $Q_M = \T$.  Thus the image of $\mathrm{Aut}(Q_M)$ in any finite group is trivial, so the conclusion of the above proposition does not hold in this case. 
\end{myeg}

\begin{myeg}
Since there are non-weakly isomorphic valuated matroids with the same underlying matroid, valuated matroid representations are more subtle than matroidal representation. For instance, see \cite[Example 2.5]{dress1992valuated} for two valuated matroids whose underlying matroid is $U_{2,4}$ but not projectively equivalent.
\end{myeg}

The following example shows that $\text{Aut}_w(M)$ is strictly smaller than $\text{Aut}(\underline{M})$ in general, but it does not have to be trivial. 

\begin{myeg}\label{example: last one}
Consider the following labeling of the bases of $U_{2,4}$:
\[
B_1=\{1,2\}, \quad B_2=\{1,3\}, \quad B_3=\{1,4\}, \quad B_4=\{2,3\}, \quad B_5=\{2,4\}, \quad B_6=\{3,4\}.
\]
Then $\sigma=(1234) \in \text{Aut}(U_{2,4})$.

Let $(M,w)$ be the valuated matroid whose underlying matroid is $U_{2,4}$, where
\[
w(B_1)=-2, \quad w(B_2)=0, \quad w(B_3)=0, \quad w(B_4)=0, \quad w(B_5)=0, \quad w(B_6)=-1.
\]
This defines a realizable valuated matroid (in \cite[Example 2.5]{dress1992valuated}).
 
We claim that $\sigma \not \in \text{Aut}_w(M)$. For the sake of contradiction, suppose that $\sigma \in \text{Aut}_w(M)$. Then, there exists a function
\[
\tau:\{1,2,3,4\} \to \mathbb{R}
\]
such that
\[
w(\sigma(B_i))=(\prod_{a \in B_i}\tau(a))w(B_i), \quad \forall i.
\]
Let $x_a:=\tau(a)$, and we use usual addition for the tropical multiplication below. Now, for $B_1$, we have
\[
w(\sigma(B_1))=w(B_4)=(x_1+x_2)+w(B_1) \iff x_1+x_2=w(B_4)-w(B_1)=2
\]
For $B_2$, we have
\[
w(\sigma(B_2))=w(B_5)=(x_1+x_3)+w(B_2) \iff x_1+x_3=w(B_5)-w(B_2)=0
\]
For $B_3$, we have
\[
w(\sigma(B_3))=w(B_1)=(x_1+x_4)+w(B_3) \iff x_1+x_4=w(B_1)-w(B_3)=-2
\]
For $B_4$, we have
\[
w(\sigma(B_4))=w(B_6)=(x_2+x_3)+w(B_4) \iff x_2+x_3=w(B_6)-w(B_4)=-1
\]
For $B_5$, we have
\[
w(\sigma(B_5))=w(B_2)=(x_2+x_4)+w(B_5) \iff x_2+x_4=w(B_2)-w(B_5)=0
\]
For $B_6$, we have
\[
w(\sigma(B_6))=w(B_3)=(x_3+x_4)+w(B_6) \iff x_3+x_4=w(B_3)-w(B_6)=1
\]
However, the resulting linear system of equations is inconsistent and hence $\sigma \not \in \text{Aut}_w(M)$. Similarly, one can check that $\sigma=(13)(24) \in \text{Aut}_w(M)$. Hence in this case $\text{Aut}_w(M)$ is a nontrivial proper subgroup of $\text{Aut}(\underline{M})$. 
\end{myeg}

\bibliography{Vectorbundle}\bibliographystyle{alpha}

\end{document}